\pdfoutput=1

\documentclass[english]{jnsao}
\usepackage[capitalise]{cleveref}
\usepackage[justification=centering]{subcaption}
\usepackage{pgfplots}
\pgfplotsset{compat=1.16}

\newlength\fheightscen \newlength\fwidthscen
\newlength\fheightcompat \newlength\fwidthcompat 
\newlength\fheightincompat \newlength\fwidthincompat 
\newlength\fheight \newlength\fwidth 

\newcommand{\bbR}{\mathbb{R}} 
\newcommand{\N}{\mathbb{N}} 
\newcommand{\bbZ}{\mathbb{Z}} 
\newcommand{\Fourier}{\mathcal{F}}
\newcommand{\Fper}{\mathcal{F}_P^\text{per}}
\newcommand{\DFT}{\mathrm{DFT}}

\newcommand{\norm}[2]{\left\Vert #1 \right\Vert_{#2}}
\newcommand{\smallnorm}[2]{{\Vert #1 \Vert}_{#2}}
\newcommand{\scalprod}[2]{\left\langle #1 \right\rangle_{#2}}
\newcommand{\smallscalprod}[2]{\langle #1 \rangle_{#2}}
\newcommand{\abs}[1]{\left\vert #1 \right\vert}

\DeclareMathOperator*{\argmin}{argmin}
\DeclareMathOperator{\ran}{ran}
\DeclareMathOperator{\spn}{span}
\DeclareMathOperator{\supp}{supp}
\DeclareMathOperator{\prox}{prox}
\DeclareMathOperator{\proj}{proj}
\DeclareMathOperator{\esssupp}{ess\,supp}

\newcommand{\X}{\mathcal{X}}
\newcommand{\Y}{\mathcal{Y}}
\newcommand{\calZ}{\mathcal{Z}}
\newcommand{\V}{\mathcal{V}}
\newcommand{\W}{\mathcal{W}}
\newcommand{\udag}{{u^\dagger}}
\newcommand{\gdag}{{g^\dagger}}
\newcommand{\gobs}{Y}
\newcommand{\bias}{\mathrm{bias}}
\newcommand{\var}{\mathrm{var}}
\newcommand{\Cov}{\mathrm{Cov}}
\newcommand{\E}[1]{\mathbb E \left[#1\right]}
\newcommand{\Var}[1]{\mathbb V \left[#1\right]}
\renewcommand{\P}[2]{\mathbb P_{#1} \left[#2\right]}
\renewcommand{\phi}{\varphi}
\renewcommand{\epsilon}{\varepsilon}
\newcommand{\Normal}{\mathcal{N}}
\newcommand{\di}{\mathrm{d}}

\newcommand{\rev}[1]{#1}

\theoremstyle{plain}
\newtheorem{thm}{Theorem}
\numberwithin{thm}{section}
\newtheorem{lem}[thm]{Lemma}
\newtheorem{cor}[thm]{Corollary}
\newtheorem{ass}{Assumption}

\theoremstyle{definition}
\newtheorem{rem}[thm]{Remark}
\newtheorem{ex}[thm]{Example}

\numberwithin{equation}{section}

\crefname{ass}{Assumption}{Assumptions}
\crefname{rem}{Remark}{Remarks}
\crefname{ex}{Example}{Examples}
\crefname{lem}{Lemma}{Lemmas}
\crefname{thm}{Theorem}{Theorems}
\crefname{cor}{Corollary}{Corollaries}

\title{Optimal regularized hypothesis testing in statistical inverse problems}
\shorttitle{Optimal regularized hypothesis testing}
\author{%
	Remo Kretschmann\thanks{Institute of Mathematics, University of Würzburg (\email{remo.kretschmann@mathematik.uni-wuerzburg.de}, \email{daniel.wachsmuth@mathematik.uni-wuerzburg.de}, \email{frank.werner@mathematik.uni-wuerzburg.de}).}
	\and Daniel Wachsmuth\footnotemark[1]
	\and Frank Werner\footnotemark[1]
}
\shortauthor{R.~Kretschmann, D.~Wachsmuth, F.~Werner}
\date{2023-10-17}

\hypersetup{
    pdftitle = {Optimal regularized hypothesis testing in statistical inverse problems},
    pdfauthor = {Remo Kretschmann, Daniel Wachsmuth, Frank Werner},
    pdfkeywords = {inverse problem, hypothesis testing, ill-posed problem, deconvolution},
}

\begin{document}

\maketitle

\begin{abstract}
Testing of hypotheses is a well studied topic in mathematical statistics. Recently, this issue has also been addressed in the context of inverse problems, where the quantity of interest is not directly accessible but only after the inversion of a (potentially) ill-posed operator. In this study, we propose a regularized approach to hypothesis testing in inverse problems in the sense that the underlying estimators (or test statistics) are allowed to be biased. Under mild source-condition type assumptions, we derive a family of tests with prescribed level $\alpha$ and subsequently analyze how to choose the test with maximal power out of this family. As one major result we prove that regularized testing is always at least as good as (classical) unregularized testing. Furthermore, using tools from convex optimization, we provide an adaptive test by maximizing the power functional, which then outperforms previous unregularized tests in numerical simulations by several orders of magnitude.
\end{abstract}

\section{Introduction}
\label{sec:intro}

\subsection{Setup}

Consider a statistical inverse problem
\begin{equation}\label{eq:model}
\gobs = T \udag   + \sigma Z
\end{equation}
where $T : \X \to \Y$ is a bounded and linear forward operator mapping between a real Banach space $\X$ and a Hilbert space $\Y$, $\udag \in \X$ is an unknown quantity of interest, $\sigma > 0$ a noise level and $Z$ a Hilbert space process on $\Y$. This means that for each $g \in \Y$ one has access to the real-valued random variable
\[
\left\langle \gobs, g\right\rangle = \left\langle T \udag, g\right\rangle + \sigma \left\langle Z, g\right\rangle,
\]
with the noise term $Z_g := \left\langle Z, g\right\rangle$, i.e. $\gobs$ can be interpreted as a random element in the \emph{algebraic} dual space of $\Y$. In the following, $\gdag := T \udag \in \Y$ denotes the ideal (but unavailable) data, and most often, $\X$ and $\Y$ are spaces of functions on some domain $\Omega \subset \bbR^d$ such as $\mathbf{L}^p \left(\Omega\right)$ or Sobolev spaces $H^s\left(\Omega\right)$.

Model \eqref{eq:model} includes the standardized Gaussian white noise model where $Z_g \sim \mathcal N \big(0, \norm{g}{\Y}^2\big)$ and $\Cov\left[ Z_{g_1}, Z_{g_2}\right] = \left\langle g_1, g_2\right\rangle$ for all $g_1, g_2 \in \Y$. In the following, we refer to this model as the Gaussian one. Model \eqref{eq:model} is a widely accepted model (see, e.g. \cite{bhmr07,mp01}), and especially the Gaussian version may serve as a prototype due to its simplicity on the one hand, but also due to its generality in view of central limit theorems and asymptotic equivalence statements on the other (cf. e.g. \cite{rsh18}).

\subsection{Estimation and Inference}
\label{sec:est_inf}

Statistical inverse problems of the form \eqref{eq:model} with ill-posed $T$ can be used to model many problems of practical interest ranging from astrophysics to cell biology, especially including (medical) imaging procedures. Consequently, estimation of $\udag$ from $Y$ as in \eqref{eq:model} has been treated extensively in the literature. Investigated methods include filter-based regularization (see, e.g. \cite{ehn96} for deterministic results and \cite{bhmr07} for results in the model \eqref{eq:model}) and regularization based on the singular value decomposition of $T$ (see, e.g. \cite{cg06,js91,mr96}), variational methods (see, e.g. \cite{hw13, miller2021}), as well as wavelet (see, e.g. \cite{jkpr04, as98, d95}) and Galerkin-approximation based methods (see, e.g. \cite{chr04}). However, in many practical applications mentioned above, not the whole function $\udag$ is of interest, but only specific features of it, such as (in the example of functions) modes, homogeneity, monotonicity, or the support. In the manuscript at hand we will therefore focus on inference for such features by means of statistical hypothesis testing.

\rev{Hypothesis testing in statistical inverse problems has been considered in the literature for global hypothesis testing problems of the form
	\begin{equation}\label{eq:ht_global}
		H_0 : \udag = 0 \qquad\text{vs.}\qquad H_1 : \udag \in B, \norm{\udag}{\X} \geq \rho
	\end{equation}
	with a smoothness class $B \subset \X$ and a radius $\rho > 0$. In this case, the situation is widely understood from a minimax viewpoint (see e.g. \cite{llm11,iss12,ilm14,mm14}), and it has furthermore been argued there that the testing problem \eqref{eq:ht_global} is equivalent (in the sense of minimax distinguishability) to 
	\[
	H_0 : \gdag = 0 \qquad\text{vs.}\qquad H_1 : \gdag \in B', \norm{\gdag}{\Y} \geq \rho'
	\]
	where $B' \subset \Y$ is another smoothness class and $\rho' > 0$ another radius. Note that this problem seems considerably simpler on first glance as no inverse problem is involved, see also \cite{emw18,pwm23} for sharp descriptions of minimax distinguishability in related models.}
	
	\rev{However, many local features of $\udag$ cannot be described by a global testing problem of the form \eqref{eq:ht_global}.} Suppose instead that there is a bounded linear functional $\varphi \in \X^*$ such that $\udag$ has a specific feature if (and maybe only if) $\left\langle \varphi, \udag\right\rangle_{\X^*\times \X} > 0$. Here and in what follows, $\X^*$ denotes the \emph{topological} dual space of $\X$, and we write $\left\langle \varphi, \udag\right\rangle_{\X^* \times \X}  = \left\langle \varphi, \udag\right\rangle := \varphi \left(\udag\right)$ to ease the notation. In fact, many interesting features such as homogeneity, support and monotonicity of functions can be described by (families of) bounded linear functionals $\varphi \in \X^*$, see, e.g. \cite{shmd13,ebd17,pwm18,depsh19}. A specific example will be discussed in Section \ref{sec:num}.

As a consequence we consider a hypothesis testing problem of the form
\begin{equation}\label{eq:hypothesis_test}
H_0 : \left\langle \varphi, \udag\right\rangle =  0\qquad \text{vs.}\qquad H_1 : \left\langle \varphi, \udag\right\rangle > 0
\end{equation}
with a linear functional in $\varphi \in \X^*$. \rev{Note, that due to freedom in the choice of $\varphi$, there is no direct connection between \eqref{eq:ht_global} and \eqref{eq:hypothesis_test}, and} that we restrict ourselves to the one-sided hypothesis testing problem \eqref{eq:hypothesis_test} for simplicity only. Most of what follows can readily be adopted to the corresponding two-sided problem where $H_1$ is replaced by $\abs{\scalprod{\varphi_B, \udag}{}} > 0$.

 In the previously mentioned works \cite{shmd13,ebd17,pwm18,depsh19} it is assumed that $\varphi \in \ran\left(T^*\right)$, i.e. there is a $\Phi_0 \in \Y^* = \Y$ such that $\varphi = T^*\Phi_0$, which yields
\begin{equation}\label{eq:lL}
	\left\langle \varphi, \udag\right\rangle_{\X^* \times \X} = \left\langle T^*\Phi_0, \udag\right\rangle_{\X^* \times \X} =\left\langle \Phi_0, T\udag\right\rangle_{\Y \times \Y}.
\end{equation}
Thus, $\left\langle \gobs,\Phi_0\right\rangle$ is a natural estimator for the desired quantity $\left\langle \varphi, \udag\right\rangle$. To design a test based on this test statistic note that
\[
\P{0}{\left\langle \gobs,\Phi_0\right\rangle> c} = \P{}{Z_{\Phi_0} > \frac{c}{\sigma}},
\]
which shows that the critical value to achieve level $\alpha\in \left(0,1\right)$ has to be chosen as $c = \sigma q_{1-\alpha} \left(\Phi_0\right)$ where $q_{1-\alpha} \left(g\right)$ is the $\left(1-\alpha\right)$-quantile of (the distribution of) $Z_{g}$ for $g \in \Y$. Concluding, the test
\begin{equation}\label{eq:test_ub}
	\Psi_{0} \left(\gobs\right) := \mathbf 1_{\left\langle \gobs,\Phi_0\right\rangle>\sigma q_{1-\alpha} \left(\Phi_0\right)} = \begin{cases} 1 & \text{if }\left\langle \gobs,\Phi_0\right\rangle>\sigma q_{1-\alpha} \left(\Phi_0\right), \\ 0 & \text{else} \end{cases}
\end{equation}
has level $\alpha \in \left(0,1\right)$ for the testing problem \eqref{eq:hypothesis_test}, i.e. $\P{0}{\Psi_0\left(\gobs\right) = 1 }  \leq \alpha$ where $\mathbb P_0$ denotes the law of $Y$ under the hypothesis $H_0$. Note that in the Gaussian model $q_{1-\alpha} \left(\Phi_0\right)= \norm{\Phi_0}{\Y} q_{1-\alpha}^{~\mathcal N}$ with the $\left(1-\alpha\right)$-quantile $q_{1-\alpha}^{~\mathcal N}$ of $\mathcal N \left(0,1\right)$.

The test $\Psi_0$ and multiple tests for families of functionals based on $\Psi_0$ have been proven to satisfy certain optimality properties (see, e.g. \cite{pwm18}), but suffer from two substantial drawbacks:
\begin{itemize}
	\item[(I1)] If $\varphi \notin \ran\left(T^*\right)$, then the above approach is not applicable. Hence, only specific properties can be tested this way.
	\item[(I2)] The computation of $\Phi_0$ involves the solution of the ill-posed equation $\varphi = T^*\Phi_0$, which for ill-posed $T^*$ implies that the norm of $\Phi_0$ and hence the critical value of $\Psi_0$ is huge.
\end{itemize}

The aim of this paper is to resolve both issues (I1) and (I2) by regularization in the sense that we allow for a bias in the estimation of $\left\langle \varphi, \udag\right\rangle_{\X^* \times \X}$ from $Y$.

\subsection{Outline}

The outline of this paper is as follows: In \cref{sec:reg_test}, we derive a whole family of test for the problem \eqref{eq:hypothesis_test} based on probe elements $\Phi \in \Y$, for which we prove that they all have prescribed level $\alpha$ under a reasonable, source-type assumption on $\udag$. The subsequent \cref{sec:opt_reg_test} is then devoted to the question which probe element $\Phi \in \Y$ should be chosen. In case of Gaussian observations, we prove that there exists an optimal $\Phi$ in the sense that the finite sample power is maximized over all tests in the previously discussed family. This $\Phi$, and hence the corresponding test, will, however, depend on the unknown $\udag$, and can hence not be accessed in practice. To resolve this issue, we develop an adaptive version in \cref{sec:adaptive}. \Cref{sec:computability} is devoted to the practical computation of both $\Phi$ and the corresponding adaptive version as solutions of a convex optimization problem. In \cref{sec:num} we discuss support inference in deconvolution problems as an example and provide numerical simulations, which show a way superior behavior of the optimal and the adaptive test compared to the unregularized test described above. \Cref{sec:conclusion} provides some discussion of the obtained results and concludes the paper.
\rev{In \cref{sec:impl_conv,sec:impl_min,sec:norm_disc}, we present details of the numerical implementation of the performed simulations.}

\section{Regularized hypothesis testing}
\label{sec:reg_test}

As the initial test $\Psi_0$ in \eqref{eq:test_ub} for the hypothesis testing problem \eqref{eq:hypothesis_test} was based on the estimator $\left\langle \gobs, \Phi_0\right\rangle$ for the feature value $\left\langle \varphi, \udag\right\rangle$, it seems natural to consider tests based on other estimators. Estimating linear functionals in statistical inverse problems is a well-studied topic, see, e.g. \cite{a86,mp02,p17}. We restrict the analysis to linear estimators here, which can always be described in terms of a probe element $\Phi \in \Y$, with the corresponding estimator $\left\langle \gobs,\Phi\right\rangle$.

\rev{\begin{ex}\label{reg_plug_in} Plug-in estimators such as $\left\langle \varphi, \hat u_\beta\right\rangle$ with spectral estimators $\hat u_\beta = q_\beta\left(T^*T\right)T^*Y$ for $\udag$, where $q_\beta(\cdot)$ is a filter and $\beta > 0$ a regularization parameter, can be expressed in the above form. The computation
\[
\left\langle \varphi, \hat u_\beta\right\rangle = \left\langle \varphi, q_\beta \left(T^*T\right) T^* Y\right\rangle  = \left\langle q_\beta \left(TT^*\right) T\varphi,  Y\right\rangle,
\]
reveals $\left\langle \varphi, \hat u_\beta\right\rangle$ to be $\left\langle \Phi_\beta, Y\right\rangle$ with $\Phi_\beta = q_\beta \left(TT^*\right) T\varphi$ being the regularized solution of $T^*\Phi = \varphi$ using the given filter and regularization parameter.
\end{ex}}

\rev{The linear estimator $\left\langle \gobs,\Phi\right\rangle$ has the bias}
\begin{equation}
\label{eq:bias}
\bias\left(\Phi\right) := \E{\left\langle \gobs,\Phi\right\rangle} - \left\langle \varphi, \udag\right\rangle = \left\langle T^*\Phi - \varphi, \udag \right\rangle_{\X^* \times \X},
\end{equation}
which clearly vanishes if $T^*\Phi = \varphi$, i.e. the test $\Psi_0$ discussed in subsection \ref{sec:est_inf} is based on an unbiased estimator. However, in inverse problems it is well-known that good estimators require a careful trade-off between bias and variance, indicating that $\left\langle \gobs, \Phi_0\right\rangle$ --- being unbiased --- is not a reasonable estimator. The variance of the general linear estimator $\left\langle \gobs,\Phi\right\rangle$ can be computed as $\var\left(\Phi\right) := \Var{\left\langle \gobs,\Phi\right\rangle} = \sigma^2\Var{Z_\Phi}$. In the Gaussian case this simplifies to $\var\left(\Phi\right)= \sigma^2 \norm{\Phi}{\Y}^2$. The estimator $\left\langle \Phi,Y\right\rangle$ gives rise to a test $\Psi_{\Phi,c}$ for \eqref{eq:hypothesis_test} defined by
\begin{equation}\label{eq:test}
\Psi_{\Phi,c} \left(\gobs\right) := \mathbf 1_{\left\langle \gobs,\Phi\right\rangle> c}
.
\end{equation}
To ensure that $\Psi_{\Phi,c}$ has level $\alpha \in\left(0,1\right)$, one should choose
\begin{equation}\label{eq:c}
c \geq c \left(\Phi, \alpha\right) := \sigma q_{1-\alpha}\left(\Phi\right) + \abs{\bias\left(\Phi\right)},
\end{equation}
as then it holds
\begin{align*}
\P{0}{\left\langle \gobs,\Phi\right\rangle > c} = \P{}{\sigma Z_\Phi + \left\langle T \udag, \Phi\right\rangle \geq c} =\P{}{Z_\Phi \geq \sigma^{-1} \left(c - \bias \left(\Phi\right)\right)} \leq \alpha.
\end{align*}
However, $\udag$ is unknown and hence are $\bias \left(\Phi\right)$ and $c \left(\Phi, \alpha\right)$ in \eqref{eq:c}. Thus, a value $c$ satisfying \eqref{eq:c} cannot be determined in practice, which shows a central problem when introducing a bias in the test statistic. Note that, if one knows a priori that $\bias \left(\Phi\right) \leq 0$, then $\Psi_{\Phi,\sigma q_{1-\alpha}\left(\Phi\right)}$ is a level $\alpha$ test for \eqref{eq:hypothesis_test} in view of \eqref{eq:c}. From this point of view, the following considerations are only necessary if the bias can be positive.

A similar problem which has been studied for some time occurs in non-parametric regression when one tries to construct (honest) confidence bands. As discussed e.g. by \cite{bbh10,pbd15}, this problem can be overcome by so-called \emph{oversmoothing}, this is ensuring that the bias is smaller than the standard deviation. In our case this corresponds to choose $\Phi \in \Y$ such that $\bias\left(\Phi\right) \ll \sigma q_{1-\alpha} \left(\Phi\right)$ in the small noise limit $\sigma \searrow 0$ and then use $c := 2 \sigma q_{1-\alpha}\left(\Phi\right)$. However, this approach might be incompatible with (I1) and does further more not allow for finite sample guarantees. Another idea is to introduced a self-similarity condition in the sense that some norm of a lower frequency part behaves similar to the same norm of a higher frequency part of $\udag$. This idea has been introduced in \cite{gn10} and inspired a series of further publications later on (see, e.g. \cite{hn11,bn13} and the references therein). More recently, this assumption became known as the \textit{polished tail condition}, see \cite{svdVvZ15} and also \cite{iw20} for a recent survey in the Bayesian context. The approach considered below shares some similarity with these ideas, but the condition posed here is much simpler to understand and also takes into account the specific structure of the hypothesis testing problem here, see Assumption \ref{ass:smooth}.

If one does not know in advance whether $\bias \left(\Phi\right) \leq 0$, then one has to upper bound this term in \eqref{eq:c}. This is possible based on \textit{a priori} information on $\udag$, which is a common paradigm in inverse problems:
\begin{ass}
	\label{ass:smooth}
	\renewcommand{\theenumi}{(\arabic{enumi})}
	\renewcommand{\labelenumi}{\theenumi}
	\begin{enumerate}
		\item \label{ass_sm_1} There is a pair of Banach spaces $\left(\V, \V'\right)$ such that
		\[
		\left\langle v,u\right\rangle_{\X^* \times \X} \leq \left\Vert v \right\Vert_{\V'}\left\Vert u\right\Vert_{\V}  \qquad\text{for all} \qquad u \in \V \cap \X, v \in \V' \cap \X^*.
		\]
		\item \label{ass_sm_2} It holds $\udag \in \V \cap \X$ with $\norm{\udag}{\V} \leq 1$.
		\item \label{ass_sm_3} $\ran T^* \subseteq \V'$ and $T^* : \Y \to \V'$ is bounded.
		\item \label{ass_sm_4} It holds $\varphi \in \overline{\ran T^*}$, where the closure is taken in $\V'$.
	\end{enumerate}
\end{ass}

Item \ref{ass_sm_1} is a rather mild requirement on the structure of $\V$ and $\V'$, which are free to be chosen so far. Item \ref{ass_sm_2} can --- to some extend --- be seen as a spectral source condition (see, e.g. \cite{ehn96}), as it requires $\udag$ to be an element with bounded norm in a smoother space $\V$. However, it also allows for other situations of interest, e.g. that $\udag$ is a density on some domain $\Omega$ or is a bounded function. Items \ref{ass_sm_3} and \ref{ass_sm_4} require compatibility of $T$ and the feature functional $\varphi$ with this information.

\begin{rem}
	In comparison with the unregularized test, whose scope was restricted to $\phi \in \ran T^*$, \cref{ass:smooth} allows to consider features described by $\phi \in \overline{\ran T^*}$.
	Let us assume that $\X$ is a Hilbert space and consider an arbitrary $\phi \in \X \cong \X^*$. In this case, we can decompose $\phi = \phi_1 + \phi_2$ into two components
	\[ \phi_1 \in \overline{\ran T^*} = (\ker T)^\perp \quad \text{and} \quad \phi_2 \in (\ran T^*)^\perp = \ker T. \]
	By doing the same with $\udag$ and using orthogonality, we can express the feature as
	\[ \scalprod{\phi,\udag}{\X} = \scalprod{\phi_1,u_1^\dagger}{\X} + \scalprod{\phi_2,u_2^\dagger}{\X}. \]
	Since the component $u_2^\dagger$ of $\udag$ lies in $\ker T$, no information about it can be obtained even from an exact measurement $T\udag$. Consequently, no information can be obtained about the contribution $\scalprod{\phi_2,u_2^\dagger}{\X}$ of $u_2^\dagger$ to the feature.
	For this reason, it is not reasonable to try to perform inference for features $\phi \in \X^* \setminus \overline{\ran T^*}$ which have a nonzero component in $(\ran T^*)^\perp$. No information about them can be obtained from the measured data.
\end{rem}

Under Assumption \ref{ass:smooth} one can provide a universal choice for the critical value $c$ that ensures a prescribed level $\alpha$, which leads to the following central result:
\begin{thm}\label{thm1}
	Let Assumption \ref{ass:smooth} hold true, $\alpha \in\left(0,1\right)$, $\Phi \in \Y$, and choose
	\begin{equation}\label{eq:c_star}
	c^* := \sigma q_{1-\alpha}\left(\Phi\right) + \left\Vert T^*\Phi - \varphi \right\Vert_{\V'}.
	\end{equation}
	Then the test $\Psi_{\Phi, c^*}$ as in \eqref{eq:test} has level at most $\alpha$ for the testing problem \eqref{eq:hypothesis_test}.
\end{thm}
\begin{proof}
	Note that under Assumption \ref{ass:smooth} one has the bias estimate
	\begin{equation}\label{eq:bias_est}
	\bias \left(\Phi\right)  = \left\langle \udag,T^*\Phi - \varphi \right\rangle_{\X \times \X^*} \leq \left\Vert \udag\right\Vert_{\V}\left\Vert T^*\Phi - \varphi \right\Vert_{\V'}  \leq\left\Vert T^*\Phi - \varphi \right\Vert_{\V'}.
	\end{equation}
	Thus $c^* \geq c\left(\Phi, \alpha\right)$ and the claim follows.
\end{proof}

With the previous theorem at hand, it is now possible to construct a variety of tests for \eqref{eq:hypothesis_test}. More precisely, every probe function $\Phi \in \Y$ gives rise to a level $\alpha$ test. However, the power $\P{1}{\Psi_{\Phi, c^*} = 1}$ of the corresponding test depends clearly on $\Phi$, and it cannot be expected to show a good performance (in the sense of a large power) for arbitrary $\Phi$.

\rev{\begin{rem} We especially find that any of the plug-in tests from \cref{reg_plug_in}, calibrated with the corresponding critical value $c^*$ in \eqref{eq:c_star}, has at most level $\alpha$ independent of the chosen regularization parameter $\beta > 0$.\end{rem}}

\section{Optimal regularized hypothesis testing}
\label{sec:opt_reg_test}

Let us now discuss the choice of the probe functional $\Phi  \in \Y$. Under Assumption \ref{ass:smooth}, the test $\Psi_{\Phi, c^*}$ with $c^*$ as in \eqref{eq:c_star} is a level $\alpha$ test, no matter how $\Phi \in\Y$ was chosen. Therefore it seems reasonable to ask for the best possible $\Phi$ in terms of the test's power, which is given as
\begin{equation}
\label{eq:pow_reg}
\begin{split}
\P{1}{\Psi_{\Phi, c^*} = 1} &= \P{1}{\left\langle \gobs,\Phi\right\rangle \geq \sigma q_{1-\alpha}\left(\Phi\right) + \norm{T^*\Phi - \varphi}{\V'}} \\
&= \P{}{\sigma Z_\Phi + \left\langle \udag,T^*\Phi\right\rangle \geq \sigma q_{1-\alpha} \left(\Phi\right) + \norm{T^*\Phi - \varphi}{\V'}} \\
&= \P{}{Z_\Phi \geq q_{1-\alpha} \left(\Phi\right) + \frac{\norm{T^*\Phi - \varphi}{\V'} - \left\langle \udag,T^*\Phi\right\rangle}{\sigma}}
\end{split}
\end{equation}
where $\mathbb P_1$ denotes the law of $Y$ under the alternative $H_1$. 
From this point onward, we restrict our analysis to the case of Gaussian white noise.
\begin{ass}
	\label{ass:gauss_noise}
	$Z$ is a Gaussian white noise process, i.e., $Z_g \sim \mathcal N \big(0, \norm{g}{\Y}^2\big)$ for all $g \in \Y$ and $\Cov\left[ Z_{g_1}, Z_{g_2}\right] = \left\langle g_1, g_2\right\rangle$ for all $g_1, g_2 \in \Y$.
\end{ass}
Here, the power of the test $\Psi_{\Phi,c^*}$ is given by
\begin{equation}
	\label{eq:pow_reg_gauss}
	\P{1}{\Psi_{\Phi, c^*} = 1} = \P{}{\frac{\scalprod{Z,\Phi}{}}{\norm{\Phi}{\Y}} \geq q_{1-\alpha}^\Normal + \frac{\norm{T^*\Phi - \varphi}{\V'} - \left\langle \udag,T^*\Phi\right\rangle_{\X \times \X^*}}{\sigma\norm{\Phi}{\Y}}}.
\end{equation}
We introduce the class of functionals $J_{y}^\W$,
\begin{equation*}
	J_{y}^\W\left(\Phi\right) := \frac{\norm{T^*\Phi - \varphi}{\V'} -  \scalprod{y,\Phi}{\W^* \times \W}}{\norm{\Phi}{\W}} \qquad \text{for all}~\Phi \in \W \setminus \{0\},
\end{equation*}
where $\W \subseteq \Y$ is a Hilbert space and $y \in \W^* \supseteq \Y^*$.
Throughout, we set $J_{y}^\W(0) := \infty$.
The optimal probe functional $\Phi \in \Y$ can then be determined by minimizing
\begin{equation}
	\label{eq:opt_phi}
	J_{T\udag}^\Y\left(\Phi\right) := \frac{\left\Vert T^*\Phi - \varphi \right\Vert_{\V'} -  \left\langle T\udag,\Phi \right\rangle_\Y}{\left\Vert \Phi \right\Vert_{\Y}}.
\end{equation}
Throughout, we make the following assumptions on the space $\W$.
\begin{ass}
	\label{ass:W}
	$(\W, \norm{\cdot}{\W})$ is a Hilbert space that is dense and continuously embedded in $\Y$.
\end{ass}
This allows especially for the setting $\W = \Y$ as in \eqref{eq:opt_phi}, but it will turn out later that for adaptive testing (i.e. testing with unknown $\udag$), smaller spaces $\W$ are necessary. Precisely, the class $\{J_y^\W\}_{y,\W}$ contains the functional $J_{T\udag}^\Y$, where we identify $T\udag$ with $\scalprod{T\udag,\cdot}{\Y}$, but also allows for data $y$ corrupted by Gaussian white noise with a suitable choice of $\W$, which will be useful later.

\begin{rem}
	If $\phi \in \ran T^*$, $T^*\Phi_0 = \phi$, and $\scalprod{\udag,\phi}{} > 0$, then it follows from \eqref{eq:pow_reg_gauss} that the unregularized test $\Psi_{\Phi_0,c_0}$ with $c_0 = \sigma q_{1 - \alpha}^\Normal \norm{\Phi_0}{\Y}$ has power
	\begin{equation}
		\label{eq:power_unreg}
		\beta(\Phi_0) = Q\left(q_\alpha^\Normal + \frac{\scalprod{\udag,\varphi}{}}{\sigma\norm{\Phi_0}{\Y}}\right) > \alpha
	\end{equation}
	where $Q$ denotes the cumulative distribution function of the standard normal distribution.
	We see that the unregularized test always has non-trivial power, but the gap to the level of significance $\alpha$ may be arbitrarily small due to the ill-posedness of $T$ and an arbitrarily large norm of $\Phi_0$ as a consequence of it.
\end{rem}

We emphasize that $J_{T\udag}^\Y$ clearly depends on the unknown quantity $\udag$. For a moment, to theoretically investigate existence and computability of the optimal probe functional $\Phi \in \Y$, we will neglect this problem and assume $\udag$ to be known. Later on we will relax this and derive a heuristic approximation $J_Y^\W$ of $J_{T\udag}^\Y$ based on known quantities only. Secondly, on first glance it seems that the functional $J_{T\udag}^\Y$ does not have any favorable structure such as convexity, and hence its minimum is in general difficult do determine. However, we will now derive a relation to a convex functional which can be minimized by standard algorithms.

As mentioned before, suppose for this subsection that $\udag$ is known. As a first step, let us prove existence of an optimal $\Phi$ in the special case $y = T\udag$.
\begin{thm}
	\label{thm:Jsol}
	Suppose Assumptions \ref{ass:smooth} and \ref{ass:W} hold and that $\left\langle \udag, \varphi\right\rangle > 0$. Then there exists a global minimum of $J_{T\udag}^\W$ and the minimum is negative.
\end{thm}
\begin{proof}
By assumption, there is a sequence $\left(v_n \right)_{n \in \N} \subset \Y$ such that $T^*v_n \to \varphi$ in $\V'$ as $n \to \infty$. By the density of $\W$ in $\Y$, we can approximate $(v_n)_{n\in\N}$ by a sequence $(w_n)_{n\in\N} \subset \W$ such that $T^*w_n \to \phi$  in $\V'$ as well.
It follows from \cref{ass:smooth} \ref{ass_sm_1} that
\[
\norm{w_n}{\W} J_{T\udag}^\W(w_n) = \left\Vert T^*w_n - \varphi \right\Vert_{\V'} -  \left\langle \udag,T^*w_n\right\rangle_{\X \times \X^*} \to - \left\langle \udag, \varphi\right\rangle_{\X \times \X^*} < 0,
\]
which implies that there exists $n \in \N$ for which $w_n\ne 0$ and $J_{T\udag}^\Y(w_n) < 0$ as well. This proves negativity of the infimum $\inf_{\Phi \in \W \setminus \{0\}} J_{T\udag}^\W\left(\Phi\right)$.

Now let $\left(\Phi_n \right)_{n \in \N} \subset \W \setminus \{0\}$ be a minimizing sequence for $J_{T\udag}^\W$. W.l.o.g.~there is $\tau > 0$ such that $J_{T\udag}^\W \left(\Phi_n\right) \leq - \tau$ for all $n \in \N$. This implies
\begin{equation}
	\label{eq:aux1}
	\left\Vert T^*\Phi_n - \varphi \right\Vert_{\V'} -  \left\langle \udag,T^*\Phi_n\right\rangle + \tau \norm{\Phi_n}{\W} \leq 0.
\end{equation}
We conclude
\begin{equation*}
	\left\Vert T^*\Phi_n - \varphi \right\Vert_{\V'} + \tau \norm{\Phi_n}{\W} \leq \left\langle \udag,T^*\Phi_n\right\rangle = \left\langle \udag,T^*\Phi_n - \varphi + \varphi\right\rangle
	\leq \left\Vert T^*\Phi_n - \varphi \right\Vert_{\V'} + \left\langle \udag,\varphi\right\rangle
\end{equation*}
using \cref{ass:smooth} \ref{ass_sm_1} and \ref{ass_sm_2}, and hence
\[
\tau \norm{\Phi_n}{\W} \leq \left\langle \udag,\varphi\right\rangle.
\]
This shows that the minimizing sequence is bounded. As $\W$ is a Hilbert space, this implies the existence of a subsequence that converges weakly towards an element $\Phi \in \W$. It follows from \eqref{eq:aux1}, \cref{ass:smooth} \ref{ass_sm_3}, and the weak lower semicontinuity of the norm that
\begin{equation*}
	\left\Vert T^*\Phi - \varphi \right\Vert_{\V'} - \left\langle \udag,T^*\Phi\right\rangle + \tau \norm{\Phi}{\W}
	\leq \liminf_{n\to\infty} \left( \left\Vert T^*\Phi_n - \varphi \right\Vert_{\V'} -  \left\langle \udag,T^*\Phi_n\right\rangle + \tau \norm{\Phi_n}{\W} \right) \leq 0,
\end{equation*}
which shows that $\Phi \neq 0$.
For any negative sequence $(a_n)_{n\in\N}$ and positive sequence $(b_n)_{n\in\N}$ we have the estimate
\begin{equation}
	\label{eq:liminf_ineq}
	\liminf_{n\to\infty} \frac{a_n}{b_n} = -\limsup_{n\to\infty} \frac{-a_n}{b_n} \ge -\limsup_{n\to\infty} (-a_n) \limsup_{n\to\infty} \frac{1}{b_n}
	= \frac{\liminf_{n\to\infty} a_n}{\liminf_{n\to\infty} b_n}.
\end{equation}
Now, it follows from the minimizing property of $(\Phi_n)_{n\in\N}$, the negativity of $J_{T\udag}^\W(\Phi_n)$ for all $n \in \N$, the positivity of $\norm{\Phi_n}{\W}$, and the weak lower semicontinuity of the norm that
\begin{align*}
	\inf_{\Phi' \in \Y} J_{T\udag}^\W(\Phi') &= \liminf_{n\to\infty} J_{T\udag}^\W(\Phi_n) \ge \frac{\liminf_{n\to\infty} \left(\norm{T^*\Phi_n - \phi}{\V'} - \scalprod{\udag,T^*\Phi_n}{}\right)}{\liminf_{n\to\infty} \norm{\Phi_n}{\W}} \\
	&\ge \frac{\norm{T^*\Phi - \phi}{\V'} - \scalprod{\udag,T^*\Phi}{}}{\norm{\Phi}{\W}} = J_{T\udag}^\W(\Phi).
\end{align*}
That is, $\Phi$ is a minimizer of $J_{T\udag}^\W$.
\end{proof}

\begin{rem}
	For $\W = \Y$, the negativity of the infimum implies that the corresponding test $\Psi_{\Phi, c^*}$ with $c^*$ as in \eqref{eq:c_star} always has a non-trivial power.
\end{rem}

\begin{thm}
	Let $\Phi^\dagger \in \Y \setminus \{0\}$ be a global minimizer of $J_{T\udag}^\Y$. Then the power $\beta(\Phi^\dagger)$ of the test $\Psi_{\Phi^\dagger, c^\dagger}$ with
	\[ c^\dagger = \sigma q_{1 - \alpha}^\Normal \norm{\Phi^\dagger}{\Y} + \norm{T^*\Phi^\dagger - \phi}{\V'} \]
	is given by
	\begin{equation}
		\label{eq:power_oracle}
		\beta(\Phi^\dagger) = Q\left(q_{\alpha}^\Normal - \frac{J_{T\udag}^\Y(\Phi^\dagger)}{\sigma}\right) > \alpha,
	\end{equation}
	and the power is maximal among all tests $\Psi_{\Phi,c^*}$ with $\Phi \in \Y$, $\Phi \neq 0$, and $c^*$ as in \eqref{eq:c_star}.
	In particular, we have $\beta(\Phi^\dagger) \ge \beta(\Phi_0)$ if $\varphi = T^*\Phi_0$ for the power $\beta(\Phi_0)$ of the unregularized test $\Psi_{\Phi_0,c_0}$ with $c_0 = \sigma q_{1 - \alpha}^\Normal \norm{\Phi_0}{\Y}$.
\end{thm}
\begin{proof}
	For any $\Phi \in \Y \setminus \{0\}$, the test $\Phi_{\Phi,c^*}$ has power
	\[ \beta(\Phi) = \P{1}{\Psi_{\Phi,c^*}(Y) = 1} = 1 - Q\left(q_{1 - \alpha}^\Normal + \frac{J_\udag(\Phi)}{\sigma}\right)
	= Q\left(q_{\alpha}^\Normal - \frac{J_\udag(\Phi)}{\sigma}\right) \]
	according to \eqref{eq:pow_reg_gauss}.
	Now, the power of $\Psi_{\Phi^\dagger,c^\dagger}$ is maximal by the choice of $\Phi^\dagger$.
	Moreover, the minimum $J_{T\udag}^\Y(\Phi^\dagger)$ is negative by \cref{thm:Jsol} and hence $\beta(\Phi^\dagger) > \alpha$ due to the monotonicity of $Q$.
\end{proof}

\rev{\begin{rem} The power of $\Psi_{\Phi^\dagger, c^*}$ is especially at least as large as the power of any of the plug-in tests from \cref{reg_plug_in}, no matter how the regularization parameter $\beta > 0$ or the filter $q_\beta(\cdot)$ is chosen.\end{rem}}

\section{Adaptive testing for unknown $\udag$}
\label{sec:adaptive}

Let us now return to the practical situation that $\udag$ is unknown. In this case, $J_{T\udag}^\Y$ in \eqref{eq:opt_phi} is also unknown and $\Phi$ cannot be found as its minimizer. A first attempt to overcome this would be to approximate
\[
J_{T\udag}^\Y \left(\Phi\right) = \frac{\left\Vert T^*\Phi - \phi \right\Vert_{\V'} - \left\langle T\udag,\Phi\right\rangle_{\Y}}{\left\Vert \Phi \right\Vert_{\Y}}
\]
by
\begin{equation}
	\label{eq:naive_approx}
	\frac{\left\Vert T^*\Phi - \phi \right\Vert_{\V'} -  \left\langle Y,\Phi\right\rangle}{\left\Vert \Phi \right\Vert_{\Y}}.
\end{equation}
However, the linear functional $\scalprod{Y,\cdot}{}$ is almost surely unbounded. Due to this lack of continuity properties, it is a priori not possible to show existence of a minimizer of the functional in \eqref{eq:naive_approx}. Constructing a test by minimizing it is therefore unfeasable.
We resolve this difficulty by restricting the functional to a dense, continuously embedded subspace $\calZ \subseteq \Y$ such that the data $Y$ almost surely is a bounded linear functional on $\calZ$.
\begin{ass}
	\label{ass:Z}
	$(\calZ, \norm{\cdot}{\calZ})$ is a Hilbert space that is dense and continuously embedded in $\Y$ such that $Z \in \calZ^*$ almost surely.
\end{ass}
For $\Phi \in \calZ \setminus \{0\}$, we can now almost surely identify this functional with
\[
\frac{\left\Vert T^*\Phi - \phi \right\Vert_{\V'} - \left\langle Y,\Phi\right\rangle_{\calZ^* \times \calZ}}{\left\Vert \Phi \right\Vert_{\Y}} = \frac{\norm{\Phi}{\calZ}}{\norm{\Phi}{\Y}} J_Y^\calZ(\Phi).
\]
In order to be able to draw a connection to a convex optimization problem later in \cref{sec:computability}, it is, however, preferable to instead choose the probe functional $\Phi \in \calZ$ as minimizer of the unweighted empirical objective functional
\[ J_Y^\calZ(\Phi) = \frac{\norm{T^*\Phi - \phi}{\V'} - \scalprod{Y,\Phi}{\calZ^* \times \calZ}}{\norm{\Phi}{\calZ}}. \]
In contrast with $J_{T\udag}^\calZ$, the existence of a minimizer of $J_Y^\calZ$ is no longer guaranteed.
\begin{thm}
	\label{sol_approx}
	Suppose that Assumptions \ref{ass:smooth} and \ref{ass:W} hold and that $\scalprod{\udag,\phi}{} > 0$. If $y \in \W^*$ satisfies
	\begin{equation}
		\label{eq:cond_data}
		\norm{y - T\udag}{\W^*} < -\frac12\min_{\Phi \in \W} J_{T\udag}^\W(\Phi),
	\end{equation}
	then a global minimum of $J_y^\W$ exists and this is negative.
\end{thm}
\begin{proof}
By \cref{thm:Jsol}, there exists a minimizer $\Phi^\dagger \in \W$ of $J_{T\udag}^\W$ and $J_{T\udag}^\W(\Phi^\dagger)$ is negative. By assumption,
\begin{equation}\label{eq_tau_lower_bound}
\tau := -J_{T\udag}^\W(\Phi^\dagger) - \norm{y - T\udag}{\W^*} \ge -J_{T\udag}^\W(\Phi^\dagger) - 2\norm{y - T\udag}{\W^*} > 0.
\end{equation}
This implies that
\begin{equation}
\label{eq:est_approx_sol}
\begin{aligned}
	\inf_{\Phi \in \W} J_y^\W(\Phi) &\le J_y^\W\left(\Phi^\dagger\right)
	= J_{T\udag}^\W\left(\Phi^\dagger\right) - \scalprod{y - T\udag, \frac{\Phi^\dagger}{\norm{\Phi^\dagger}{\W}}}{\W^* \times W} \\
	&\le J_{T\udag}^\W\left(\Phi^\dagger\right) + \norm{y - T\udag}{\W^*} = -\tau < 0.
\end{aligned}
\end{equation}

Now we proceed similar as in the proof of \cref{thm:Jsol}. Let $\left(\Phi_n \right)_{n \in \N} \subset \W \setminus \{0\}$ be a minimizing sequence. W.l.o.g.~we can choose it such that $J_{y}^\W \left(\Phi_n\right) \leq -\tau$ for all $n \in \N$. This implies
\begin{equation}
	\label{eq:aux2}
	\left\Vert T^*\Phi_n - \varphi \right\Vert_{\V'} - \left\langle y,\Phi_n\right\rangle + \tau \norm{\Phi_n}{\W} \leq 0.
\end{equation}
We conclude
\begin{align*}
	\norm{T^*\Phi_n - \varphi}{\V'} + \tau\norm{\Phi_n}{\W} &\leq \scalprod{y,\Phi_n}{}
	= \scalprod{y - T\udag,\Phi_n}{} + \scalprod{\udag,T^*\Phi_n - \phi}{} + \scalprod{\udag,\phi}{} \\
	&\le \norm{y - T\udag}{\W^*} \norm{\Phi_n}{\W} + \norm{T^*\Phi_n - \phi}{\V'} + \scalprod{\udag,\phi}{}
\end{align*}
using \cref{ass:smooth} \ref{ass_sm_1} and \ref{ass_sm_2}, and hence
\[
\left(\tau - \norm{y - T\udag}{\W^*}\right) \norm{\Phi_n}{\W} \leq \left\langle \udag,\varphi\right\rangle.
\]
This shows that the minimizing sequence is bounded since $\tau - \norm{y - T\udag}{\W^*} > 0$ by assumption, compare \eqref{eq_tau_lower_bound}. As $\W$ is a Hilbert space, this implies the existence of a subsequence that converges weakly towards an element $\Phi \in \W$. It follows from \eqref{eq:aux2}, \cref{ass:smooth} \ref{ass_sm_3}, and the lower semicontinuity of the norm that
\begin{equation*}
	\left\Vert T^*\Phi - \varphi \right\Vert_{\V'} - \left\langle y,\Phi\right\rangle + \tau \norm{\Phi}{\W}
	\leq \liminf_{n\to\infty} \left( \left\Vert T^*\Phi_n - \varphi \right\Vert_{\V'} -  \left\langle y,\Phi_n\right\rangle + \tau \norm{\Phi_n}{\W} \right) \leq 0,
\end{equation*}
which shows that $\Phi \neq 0$.
Now, it follows from \eqref{eq:liminf_ineq} that
\begin{align*}
	\inf_{\Phi' \in \Y} J_y^\W(\Phi') &= \liminf_{n\to\infty} J_y^\W(\Phi_n) \ge \frac{\liminf_{n\to\infty} \left(\norm{T^*\Phi_n - \phi}{\V'} - \scalprod{y,\Phi_n}{}\right)}{\liminf_{n\to\infty} \norm{\Phi_n}{\W}} \\
	&\ge \frac{\norm{T^*\Phi - \phi}{\V'} - \scalprod{y,\Phi}{}}{\norm{\Phi}{\W}} = J_y^\W(\Phi).
\end{align*}
That is, $\Phi$ is a minimizer of $J_{y}^\W$.
\end{proof}

\begin{cor}
	\label{prob_sol_approx}
	Suppose that Assumptions \ref{ass:smooth} and \ref{ass:Z} hold and that $\scalprod{\udag,\phi}{} > 0$. 
	Then the probability that $J_Y^\calZ$ has a minimum and this minimum is negative is bounded from below by
	\[ \P{}{\norm{Z}{\calZ^*} < -\frac{1}{2\sigma} \min_{\Phi \in \calZ \setminus \{0\}} J_{T\udag}^\calZ(\Phi)}. \]
	In particular, this probability converges to $1$ as $\sigma \to 0$.
\end{cor}
\begin{proof}
	By \eqref{eq:model},
	\[ \sigma\norm{Z}{\calZ^*} = \norm{Y - T\udag}{\calZ^*} < -\frac12\min_{\Phi \in \calZ \setminus \{0\}} J_{T\udag}^\calZ(\Phi) \]
	with probability
	\[ \P{}{\norm{Z}{\calZ^*} < -\frac{1}{2\sigma} \min_{\Phi \in \calZ} J_{T\udag}^\calZ(\Phi)}. \]
	In this case, a minimum of $J_Y^\calZ$ exists by \cref{sol_approx}.
\end{proof}

\begin{rem}
	For a separable Hilbert space $\Y$, a dense, continuously embedded subspace $\calZ$ upon which the Gaussian white noise $Z$ can a.s. be identified with a bounded linear functional can be constructed using any orthonormal basis $f_k$ of $\Y$ and any positive square-summable sequence $(\omega_k)_{k\in\N}$ by
	\[ \calZ = \left\{ z \in \Y: \sum_{k=1}^\infty \omega_k^{-2} \abs{\scalprod{z,f_k}{\Y}}^2 < \infty \right\}. \]
	Its dual space is then given by
	\[ \calZ^* = \left\{ z' \in L(\Y,\bbR): \sum_{k=1}^\infty \omega_k^{2} \abs{\scalprod{z',f_k}{}}^2 < \infty \right\}. \]
	By the independence of the evaluations of $Z$ in $f_k$ and Jensen's inequality, we then have
	\[ \E{\norm{Z}{\calZ^*}}^2 \le \E{\norm{Z}{\calZ^*}^2} = \sum_{k=1}^\infty \omega_k^2\E{\abs{\scalprod{Z,f_k}{}}^2} = \sum_{k=1}^\infty \omega_k^2 < \infty, \]
	so that $Z \in \calZ^*$ almost surely, see Appendix 7.4 in \cite{Nickl2020Bernstein}.
	Note that here, the boundedness of $Z$ on $\calZ$ follows from the nuclearity of the embedding $\calZ \hookrightarrow \Y$.
	By Theorem 2.1.20 in \cite{gn2015}, $Z$ moreover satisfies the concentration inequality
	\[ \P{}{\abs{\norm{Z}{\calZ^*} - \E{\norm{Z}{\calZ^*}}} > \tau} \le 2\exp \left(-\frac{\tau^2}{2\pi^2 \sup_{k\in\N} \omega_k}\right), \]
	which can be used to bound the probability that $J_Y^\calZ$ has a negative minimum from below.
\end{rem}

A problem that arises with the approximation $J_Y^\calZ$ is that if we use the data $Y$ to define $\Phi = \Phi(Y)$ as the minimizer of $J_Y^\calZ$, then the test $\Psi_{\Phi,c^*}$ as constructed in Theorem \ref{thm1} does no longer have level $\alpha$. The reason for this is that due to the dependence of $\Phi$ on $Y$, the expectation of the estimator $\scalprod{Y,\Phi(Y)}{}$ is no longer equal to $\scalprod{T\udag,\Phi}{}$, and thus its bias is no longer given by \eqref{eq:bias}. We circumvent this problem by utilizing two independent data samples $Y_1 = T\udag + \sigma Z_1$ and $Y_2 = T\udag + \sigma Z_2$, the first one to construct the test, i.e., to choose $\Phi$ as a minimizer of $J_{Y_1}^\calZ$, while applying the test to the second one, i.e., evaluating $\Psi_{\Phi,c^*}(Y_2)$.

As $J_{Y_1}^\calZ$ does not necessarily possess a minimum, we consider the adapted test
\begin{equation}
	\label{eq:def_Psi_star}
	\Psi^*(Y_2;Y_1) := \begin{cases}
		\Psi_{\Phi,c}(Y_2) & \text{if}~J_{Y_1}^\calZ~\text{has a global minimizer}~\Phi\in\calZ, \\
		0 & \text{otherwise},
	\end{cases}
\end{equation}
where
\[ c := \sigma q_{1 - \alpha}^\Normal\norm{\Phi}{\Y} + \norm{T^*\Phi - \phi}{\V'} \]
for some $\alpha \in (0,1)$.

\begin{thm}
	\label{thm:pow_adapt_test}
	Suppose that Assumptions \ref{ass:smooth} and \ref{ass:Z} hold, that $\scalprod{\udag,\phi}{} > 0$ and $\alpha \in (0,1)$.
	Then the test $\Psi^*$ as defined in \eqref{eq:def_Psi_star} has level $\alpha$, and its power $\beta(\udag)$ is bounded from below by
	\begin{align*}
		\underline{\beta}(\udag) &:= \sup_{\tau \ge 0} Q\left(q_{\alpha}^\Normal + \frac{\tau}{\sigma C_\calZ}\right) \P{}{\norm{Z_1}{\calZ^*} < \frac{1}{2\sigma}\left(-\min_{\Phi \in \calZ} J_{T\udag}^\calZ(\Phi) - \tau\right)},
	\end{align*}
	where $Z_1$ is a standard Gaussian white noise process on $\Y$ and $C_\calZ > 0$ denotes the norm of the embedding $\calZ \hookrightarrow \Y$.
	In particular, $\beta(\udag) \to 1$ as $\sigma \to 0$.
\end{thm}
\begin{proof}
By the independence of $Y_1$ and $Y_2$ and the choice of $c$, the test $\Psi^*$ has level
\begin{equation*}
	\P{0}{\Psi^* = 1} = \P{0}{\Psi_{\Phi,c}(Y_2) = 1 \big\vert J_{Y_1}^\calZ~\text{has a minimum}}\P{0}{J_{Y_1}^\calZ~\text{has a minimum}}
	\le \alpha \cdot 1.
\end{equation*}
By \cref{thm:Jsol}, there exists a minimizer $\Phi^\dagger \in \calZ$ of $J_{T\udag}^\calZ$ and $J_{T\udag}^\calZ(\Phi^\dagger)$ is negative.
We express the power of $\Psi^*$ as
\begin{align*}
	\beta(\udag) = \P{1}{\Psi^* = 1} &= \P{1}{\Psi_{\Phi,c}(Y_2) = 1 \Bigg\vert \norm{Y_1 - T\udag}{\calZ^*} < -\frac{J_{T\udag}^\calZ(\Phi^\dagger) + \tau}{2}} \\
	&\quad \cdot\P{1}{\norm{Y_1 - T\udag}{\calZ^*} < -\frac{J_{T\udag}^\calZ(\Phi^\dagger) + \tau}{2}}
\end{align*}
for any $\tau \ge 0$.
By \eqref{eq:model}, we have $\norm{Y_1 - T\udag}{\calZ^*} = \sigma \norm{Z_1}{\calZ^*}$, where $Z_1$ is a standard Gaussian white noise process on $\Y$.
If
\[ \norm{Y_1 - T\udag}{\calZ^*} < -\frac{J_{T\udag}^\calZ(\Phi^\dagger) + \tau}{2}, \]
then $J_{Y_1}^\calZ$ has a global minimizer $\Phi \in \calZ$ by \cref{sol_approx} and
\begin{align*}
	J_{T\udag}^\calZ(\Phi) &= J_{Y_1}^\W(\Phi) + \scalprod{Y_1 - T\udag, \frac{\Phi}{\norm{\Phi}{\calZ}}}{} \le J_{Y_1}^\calZ(\Phi^\dagger) + \scalprod{Y_1 - T\udag, \frac{\Phi}{\norm{\Phi}{\calZ}}}{} \\
	&= J_{T\udag}^\calZ(\Phi^\dagger) + \scalprod{Y_1 - T\udag, \frac{\Phi}{\norm{\Phi}{\calZ}} - \frac{\Phi^\dagger}{\norm{\Phi^\dagger}{\calZ}}}{\calZ^* \times \calZ} \\
	&\le J_{T\udag}^\calZ(\Phi^\dagger) + 2\norm{Y_1 - T\udag}{\calZ^*} \le -\tau.
\end{align*}
We have $\Psi_{\Phi,c}(Y_2) = 1$ if and only if
\begin{equation*}
	 \frac{\scalprod{Z_2,\Phi}{}}{\norm{\Phi}{\Y}} \ge q_{1 - \alpha}^\Normal + \frac{J_{T\udag}^\Y(\Phi)}{\sigma} = q_{1 - \alpha}^\Normal + \frac{J_{T\udag}^\calZ(\Phi)\norm{\Phi}{\calZ}}{\sigma\norm{\Phi}{\Y}}.
\end{equation*}
Thus, it follows that
\begin{align*}
	&\P{1}{\Psi_{\Phi,c}(Y_2) = 1 \Bigg\vert \norm{Y_1 - T\udag}{\calZ^*} < -\frac{J_{T\udag}^\calZ(\Phi^\dagger) + \tau}{2}} \\
	&= \P{}{\frac{\scalprod{Z_2,\Phi}{}}{\norm{\Phi}{\Y}} \ge q_{1 - \alpha}^\Normal + \frac{J_{T\udag}^\calZ(\Phi)\norm{\Phi}{\calZ}}{\sigma\norm{\Phi}{\Y}} \Bigg\vert \norm{Y_1 - T\udag}{\calZ^*} < -\frac{J_{T\udag}^\calZ(\Phi^\dagger) + \tau}{2}} \\
	&\ge \P{}{\frac{\scalprod{Z_2,\Phi}{}}{\norm{\Phi}{\Y}} \ge q_{1 - \alpha}^\Normal - \frac{\tau}{\sigma C_\calZ} \Bigg\vert \norm{Y_1 - T\udag}{\calZ^*} < -\frac{J_{T\udag}^\calZ(\Phi^\dagger) + \tau}{2}} \\
	&= 1 - Q\left(q_{1 - \alpha}^\Normal - \frac{\tau}{\sigma C_\calZ}\right)
	= Q\left(-q_{1 - \alpha}^\Normal + \frac{\tau}{\sigma C_\calZ}\right).
\end{align*}
This yields
\begin{align*}
	\beta(\udag) &\ge Q\left(q_{\alpha}^\Normal + \frac{\tau}{\sigma C_\calZ}\right) \P{}{\norm{Z_1}{\calZ'} < \frac{-J_{T\udag}^\calZ(\Phi^\dagger) - \tau}{2\sigma}}
\end{align*}
for every $\tau \ge 0$.
\end{proof}

\section{Computability}
\label{sec:computability}

Next, we turn toward computability of the minimizer $\Phi \in \Y$ of $J_Y^\W$. Therefore we introduce the class of functionals $\hat{J}_{y}^\W$: $\W \times \bbR \to \bbR \cup \{\infty\}$,
\[
\hat{J}_{y}^\W \left(e,s\right) := \norm{T^*e - s \varphi}{\V'} - \left\langle y, e\right\rangle_{\W^* \times \W},
\]
where $\W$ is a Hilbert space that is dense and continuously embedded in $\Y$, and $y \in \W^*$.
If $\Phi \in \W$, $\Phi \neq 0$, then we have the relation
\[
J_{y}^\W \left(\Phi\right) = \hat{J}_{y}^\W \left(\frac{\Phi}{\norm{\Phi}{\W}}, \frac{1}{\norm{\Phi}{\W}}\right).
\]
Moreover, $\hat{J}_y^\W$ is positively homogeneous.
We will now show how the global minimizers of $J_y^\W$ are related to the global solutions of the convex problem
\begin{equation}
	\label{eq:approx_conv_prob}
	\min \hat{J}_y^\W(e,s) \qquad \text{subject to} \qquad \norm{e}{\W} \le 1, \quad s \ge 0.
\end{equation}
This problem is equivalent to the unconstrained problem of minimizing $\hat{J}_y^\W + \delta_{B_1^\W \times \bbR_+}$, where $\delta_S$ denotes the indicator function of a set $S$, i.e., $\delta_S(x) = 0$ if $x \in S$ and $\delta_S(x) = \infty$ otherweise, $B_1^\W$ denotes the closed unit ball in $\W$, and $\bbR_+ := [0,\infty)$.

\begin{lem}
\label{lem:surr_cont}
	Suppose \cref{ass:W}, let $U^\W := B_1^\W \times \bbR_+$ and $y \in W^*$. 
	Then the functional $\hat{J}_y^\W + \delta_{U^\W}$: $\W \times \bbR \to \bbR \cup \{\infty\}$ is convex, lower semicontinuous, and coercive.
\end{lem}
\begin{proof}
For all $e_1, e_2 \in \W$, $s_1, s_2 \in \bbR$, and $\lambda \in [0,1]$, we have
\begin{align*}
	&\hat{J}_y^\W(\lambda e_1 + (1 - \lambda) e_2, \lambda s_1 + (1 - \lambda) s_2) \\
	&= \norm{\lambda\left(T^*e_1 - s_1\phi\right) + (1 - \lambda)\left(T^*e_2 - s_2\phi\right)}{\V'}
	- \scalprod{y, \lambda e_1 + (1 - \lambda)e_2}{} \\
	&\le \lambda\norm{T^*e_1 - s_1\phi}{\V'} + (1 - \lambda)\norm{T^*e_2 - s_2\phi}{\V'}
	- \lambda\scalprod{y, e_1}{} - (1 - \lambda)\scalprod{y, e_2}{} \\
	&= \lambda\hat{J}_y^\W(e_1,s_1) + (1 - \lambda)\hat{J}_y^\W(e_2,s_2)
\end{align*}
by the triangle inequality. Now, $\hat{J}_y^\W + \delta_{U^\W}$ is convex as the sum of convex functions.

The functional $\hat{J}_y^\W$ is continuous since the mapping $\W \times \bbR \to \V'$, $(e,s) \mapsto T^*e - s\varphi$ is continuous by Assumption \ref{ass:smooth} (3) and (4) and the continuity of the embedding $\W \hookrightarrow \Y$, and since the linear functional $e \mapsto \scalprod{y,e}{\W^* \times \W}$ is bounded. The indicator function $\delta_{U^\W}$ is lower semicontinuous since $U^\W$ is convex and closed in $\W$. Consequently, $\hat{J}_y^\W + \delta_{U^\W}$ is lower semicontinuous as the sum of two lower semicontinuous functions.

Let $(e_n,s_n)_{n\in\N}$ be a sequence in $\W \times \bbR$ with $\norm{e_n}{\W} \to \infty$ or $s_n \to \infty$.
If $(e_n)_{n\in\N}$ is bounded, then $s_n$ is not. In this case, $(T^*e_n)_{n\in\N}$ is bounded in $\V'$ by Assumption \ref{ass:smooth} (3), which yields
\begin{equation*}
	\hat{J}_y^\W(e_n,s_n) = \norm{T^*e_n - s_n\varphi}{\V'} - \scalprod{y, e_n}{\W^* \times \W}
	\ge s_n\norm{\varphi}{\V'} - \norm{T^*e_n}{\V'} - \norm{y}{\W^*}\norm{w}{\W} \to \infty
\end{equation*}
as $n \to \infty$.
If, on the other hand, $(e_n)_{n\in\N}$ is unbounded, then $\delta_{U^\W}(e_n,s_n) \to \infty$. Therefore, $\hat{J}_y^\W + \delta_{U^\W}$ is coercive.
\end{proof}

\begin{thm}
	\label{thm:sol_approx_conv}
	Suppose that Assumptions \ref{ass:smooth} and \ref{ass:W} hold and let $y \in \W^*$.
	Then a solution $(e,s) \in \W \times \bbR$ of \eqref{eq:approx_conv_prob} exists.
\end{thm}
\begin{proof}
The existence of a minimizer of $\hat{J}_{T\udag}^\W + \delta_{U^\W}$ follows, using the direct method, from its convexity, lower semicontinuity, and coercivity established in Lemma \ref{lem:surr_cont}, see, e.g., \cite[Proposition II.1.2]{EkelandTemam1976}.
\end{proof}

\begin{thm}
	\label{thm:hatJ_sol}
	Suppose that Assumptions \ref{ass:smooth} and \ref{ass:W} hold, that $\scalprod{\udag,\phi}{} > 0$, and let $y \in \W^*$.
	If
	\begin{equation}
		\label{eq:conv_surr_cond}
		\norm{y - T\udag}{\W^*} < -\frac12\min_{(e,s) \in U^\W} \hat{J}_{T\udag}^\W(e,s),
	\end{equation}
	then every solution $(e,s) \in \W \times \bbR$ of \eqref{eq:approx_conv_prob} satifies
	$\norm{e}{\W} = 1$, $s > 0$, and $\hat{J}_y^\W(e,s) < 0$.
	In particular, this is the case for $y = T\udag$.
\end{thm}
\begin{proof}
By \cref{thm:sol_approx_conv}, a minimizer $(e^\dagger,s^\dagger) \in \W \times \bbR$ of $\hat{J}_{T\udag}^\W + \delta_{U^\W}$ exists.
First, we show that $\hat{J}_{T\udag}^\W(e^\dagger,s^\dagger)$ is negative.
By assumption, there exists a sequence $\left(v_n \right)_{n \in \N} \subset \Y$ such that $T^*v_n \to \varphi$ as $n \to \infty$. By the density of $\W$ in $\Y$, we can approximate $(v_n)_{n\in\N}$ by a sequence $(w_n)_{n\in\N} \subset \W$ such that $T^*w_n \to \phi$ as well.
This yields
\[
\norm{w_n}{\W} J_{T\udag}^\W(w_n) = \left\Vert T^*w_n - \varphi \right\Vert_{\V'} -  \left\langle \udag,T^*w_n\right\rangle_{\X \times \X^*} \to - \left\langle \udag, \varphi\right\rangle_{\X \times \X^*} < 0,
\]
using \cref{ass:smooth} \ref{ass_sm_3}. In particular, there exists $n \in \N$ such that the left hand side is negative.
It follows that
\[ \hat{J}_{T\udag}^\W(e^\dagger,s^\dagger) \le \hat{J}_{T\udag}^\W(e_n,s_n) = \norm{T^*e_n - s_n\varphi}{\V'} - \scalprod{\udag,T^*e_n}{\X} < 0 \]
for $e_n := w_n/\norm{w_n}{\W}$ and $s_n := 1/\norm{w_n}{\W}$.
Now, we have
\begin{equation*}
	\tau := -\hat{J}_{T\udag}^\W\left(e^\dagger,s^\dagger\right) - \norm{y - T\udag}{\W^*} > 0
\end{equation*}
by assumption.
By the optimality of $(e,s)$ for $\hat{J}_y^\W$, this implies
\begin{align*}
	\hat{J}_y^\W(e,s) &\le \hat{J}_y^\W\left(e^\dagger,s^\dagger\right)
	= \hat{J}_{T\udag}^\W\left(e^\dagger,s^\dagger\right) - \scalprod{y - T\udag,e^\dagger}{\W^* \times \W} \\
	&\le \hat{J}_{T\udag}^\W\left(e^\dagger,s^\dagger\right) + \norm{y - T\udag}{\W^*} = -\tau < 0.
\end{align*}
On the other hand, we have
\begin{align*}
	\hat{J}_y^\W(e,s) &\ge \norm{\udag}{\V}\norm{T^*e - s\phi}{\V'} - \scalprod{y,e}{}
	\ge \scalprod{\udag,T^*e - s\phi}{} - \scalprod{y,e}{} \\
	&= \scalprod{T\udag - y,e}{} - s\scalprod{\udag,\phi}{}
	\ge -\norm{y - T\udag}{\W^*} - s\scalprod{\udag,\phi}{}
\end{align*}
by \cref{ass:smooth} (1) and (2).
By \eqref{eq:conv_surr_cond}, this implies
\[ s \ge \frac{\tau - \norm{y - T\udag}{\W^*}}{\scalprod{\udag,\phi}{}} > 0. \]
We have $e \neq 0$ because assuming that $e = 0$ and using \eqref{eq:conv_surr_cond} leads to
\[ 0 > \hat{J}_y^\W(0,s) = s\norm{\phi}{\V'} \ge 0, \]
a contradiction.
We see that $\norm{e}{\calZ}^{-1}(e,s)$ is feasable for \eqref{eq:approx_conv_prob} and satisfies
\[ \hat{J}_y^\W(e,s) \le \hat{J}_y^\W\left(\norm{e}{\calZ}^{-1}e,\norm{e}{\calZ}^{-1}s\right) = \norm{e}{\calZ}^{-1}\hat{J}_y^\W(e,s). \]
Now it follows from the negativity of $\hat{J}_y^\W(e,s)$ that $\norm{e}{\calZ} = 1$.
\end{proof}

Now, we relate the constrained minimizers of $\hat{J}_y^\W$ to the global minimizers of $J_y^\W$.
\begin{thm}
	\label{thm:rel_approx_sol}
	Suppose that Assumptions \ref{ass:smooth} and \ref{ass:W} hold, that $\scalprod{\udag,\phi}{\X} > 0$ and $y \in \W^*$. If $(e,s) \in \W \times \bbR$ is a solution of \eqref{eq:approx_conv_prob}, $e \neq 0$, and $s > 0$, then $s^{-1}e$ is a global minimizer of $J_y^\W$.
	Conversely, if $\Phi \in \W$ is a minimizer of $J_y^\W$, then
	\[ \hat{J}_y^\W\left(\frac{\Phi}{\norm{\Phi}{\W}}, \frac{1}{\norm{\Phi}{\W}}\right) = \min \left\{\hat{J}_y^\W(e,s): e \in \W, \norm{e}{\W} = 1, s > 0\right\}. \]
\end{thm}
\begin{proof}
In the first case, we have
\begin{align*}
	J_y^\W\left(s^{-1}e\right) &= \frac{1}{\norm{e}{\W}} \hat{J}_y^\W(e,s)
	\le \frac{1}{\norm{e}{\W}} \hat{J}_y^\W\left(\frac{\Phi\norm{e}{\W}}{\norm{\Phi}{\W}}, \frac{\norm{e}{\W}}{\norm{\Phi}{\W}}\right) \\
	&= \hat{J}_y^\W\left(\frac{\Phi}{\norm{\Phi}{\W}}, \frac{1}{\norm{\Phi}{\W}}\right)
	= J_y^\W(\Phi)
\end{align*}
for all $\Phi \in \W \setminus \{0\}$ by the optimality of $(e,s)$.
In the second case, we have
\begin{align*}
	\hat{J}_y^\W\left(\frac{\Phi}{\norm{\Phi}{\W}}, \frac{1}{\norm{\Phi}{\W}}\right)
	&= J_y^\W(\Phi)
	\le J_y^\W\left(s^{-1}e\right)
	= \hat{J}_y^\W(e,s)
\end{align*}
for all $e \in \W$ with $\norm{e}{\W} = 1$ and $s > 0$ by the optimality of $\Phi$.
\end{proof}

\begin{cor}
	Let the assumptions of \cref{thm:rel_approx_sol} hold.
	If $\Phi \in \W$ is a minimizer of $J_{T\udag}^\W$, then 
	\[ \left(\frac{\Phi}{\norm{\Phi}{\W}},\frac{1}{\norm{\Phi}{\W}}\right)\]
	minimizes $\hat{J}_{T\udag}^\W + \delta_{U^\W}$.
\end{cor}
\begin{proof}
	This follows immediately from Theorems \ref{thm:hatJ_sol} and \ref{thm:rel_approx_sol}.
\end{proof}

\begin{rem}
	Note that $(\norm{\Phi}{\W}^{-1}\Phi,\norm{\Phi}{\W}^{-1})$ is in general not a minimizer of $\hat{J}_y^\W + \delta_{U^\W}$ if $\Phi$ is a minimizer of $J_y^\W$. This only holds if all minimizers $(e,s)$ of $\hat{J}_y^\W$ satisfy $\norm{e}{\W} = 1$ and $s > 0$.
\end{rem}

\section{Numerical simulations}
\label{sec:num}

\setlength{\fwidthscen}{4.2cm}
\setlength{\fheightscen}{3.2cm}
\setlength{\fwidthcompat}{4.2cm}
\setlength{\fheightcompat}{3.2cm}
\setlength{\fwidthincompat}{4.2cm}
\setlength{\fheightincompat}{3.2cm}
\setlength{\fwidth}{4.0cm}
\setlength{\fheight}{3.2cm}
\newcommand\sigmamincompat{3e-6}
\newcommand\sigmaminincompat{2.2627e-8}
\newcommand\sigmaminprob{1e-3}
\definecolor{mycolor1}{rgb}{0.00000,0.44700,0.74100}
\definecolor{mycolor2}{rgb}{0.85000,0.32500,0.09800}
\definecolor{mycolor3}{rgb}{0.92900,0.69400,0.12500}

In this section we will now perform a numerical case study to investigate the behavior of the optimal test $\Psi_{\Phi^\dagger, c^\dagger}$ and the adaptive test $\Psi^*$ compared to the unregularized test $\Psi_0$. As example for the forward operator we consider a convolution.

\subsection{Problem set-up and considered scenarios}

\rev{The convolution between two functions $h \in L^1(\bbR)$ and $u \in L^2(\bbR)$ is defined by
\begin{equation*}
	(h * u)(x) = \int_{\bbR} h(x - z) u(z) \di z \quad \text{for all}~x \in \bbR.
\end{equation*}
The Fourier transform of $h \in L^1(\bbR)$ is, moreover, defined by
\begin{equation*}
	(\Fourier h)(\xi) = \int_{\bbR} h(x) \exp \left(2\pi i\scalprod{x,\xi}{}\right) \di x \quad \text{for all}~\xi \in \bbR.
\end{equation*}
We consider the convolution operator $Tu := h * u$ on $L^2(\bbR)$ associated with a kernel $h \in L^1(\bbR)$ which is defined in terms of its Fourier transform
\begin{equation}
	\label{eq:def_h}
	(\Fourier h)(\xi) = \left(1 + \rev{b^2}\xi^2\right)^{-a} \quad \text{for all}~\xi \in \bbR,
\end{equation}
where $a \ge \frac12$ and $b = 0.06$. Specifically, we will consider the cases $a = 2$ and $a = 4$.}

\begin{figure}[htp]
	\centering
	\begin{tikzpicture}[baseline]
	\begin{axis}[%
		width=15cm,
		height=3.2cm,
		scale only axis,
		xmin=-0.5,
		xmax=0.5,
	]
		\addplot[color of colormap=250 of viridis] table [x index=0, y index=1] {data/kernel_a2.dat};
		\label{plot:kernel_a2}
		\addplot[color of colormap=750 of viridis] table [x index=0, y index=1] {data/kernel_a4.dat};
		\label{plot:kernel_a4}
	\end{axis}
\end{tikzpicture}%
	\caption{The convolution kernel $h$ for $a = 2$ (\ref{plot:kernel_a2}) and $a = 4$ (\ref{plot:kernel_a4}).}
	\label{fig:kernel}
\end{figure}
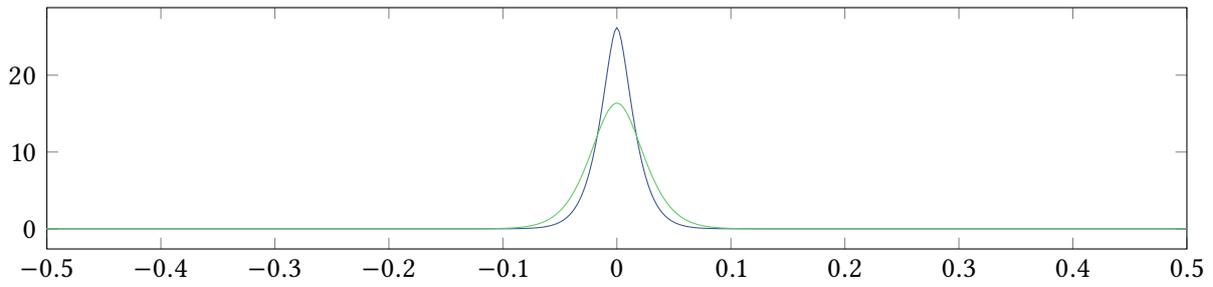

\rev{The convolution between $h$ and $u$ can be approximated by the \emph{periodic convolution}
\[ (\tilde{h} *_P \tilde{u})(x) = \int_{-\frac{P}{2}}^{\frac{P}{2}} \tilde{h}(x - z)\tilde{u}(z) \di z, \quad x \in \bbR, \]
between the $P$-\emph{periodization} $\tilde{h}$ of $h$,
\begin{equation*}
	\tilde{h}(x) := h_{\text{per},P}(x) := \sum_{l \in \bbZ} h(x + lP) \quad \text{for all}~x \in \bbR,
\end{equation*}
and the $P$-periodization $\tilde{u} := u_{\text{per},P}$ of $u$.
In the following, we periodize with $P := 2$ and consider the periodic convolution operator $\widetilde{T}\tilde{u} := \tilde{h} *_P \tilde{u}$ on 
\[ \X := \Y := L^2(-1,1) \]
associated with the periodized kernel $\tilde{h} = h_{\text{per},P} \in L^1(-1,1)$. Throughout, we moreover assume that $\esssupp \udag \subseteq [-\frac12,\frac12]$.}

\rev{We discretize the problem using a uniform grid of size $N = 1024$ and compute the discrete convolution using a fast Fourier transform. For more details on the implementation of the convolution see \cref{sec:impl_conv}.}

Let $\Fper$: $L^2(-P/2,P/2) \to \ell^2(\bbZ)$ denote the \emph{periodic Fourier transform}
\[ (\Fper f)(k) := \widehat{f}(k) := P^{-d} \int_{-\frac{P}{2}}^{\frac{P}{2}} f(x) \exp\left(-\frac{2\pi i}{P}\scalprod{k,x}{}\right) \di x \]
for all $k \in \bbZ^d$.
For $t \ge 0$, we define the Sobolev spaces
\[ H^t(-1,1) = \left\{f \in L^2(-1,1): k \mapsto (1 + k^2)^\frac{t}{2}\widehat{f}(k) \in \ell^2(\bbZ)\right\}, \]
equipped with the norm
\[ \norm{f}{H^t} = P^\frac12 \left(\sum_{k \in \bbZ} \left(1 + k^2\right)^t \rev{\abs{\widehat{f}(k)}^2}\right)^\frac12. \]

\begin{lem}
	\label{ran_T_H4}
	The operator $\widetilde{T}$ satisfies $\ran \widetilde{T} \subseteq \rev{H^{2a}(-1,1)}$ and $\widetilde{T}$: $L^2(-1,1) \to \rev{H^{2a}(-1,1)}$ is bounded.
\end{lem}
\begin{proof}
	By \rev{the periodic convolution theorem and} Poisson's summation formula, we have
	\begin{equation*}
		\begin{split}
			\Fper\left(\widetilde{T}\tilde{u}\right)(k) &= \Fper\left(\tilde{h} * \tilde{u}\right)(k) = P\Fper h_{\text{per},P}(k) \cdot \Fper \tilde{u}(k) \\
			&= (\Fourier h)\left(\frac{k}{P}\right) \cdot \Fper\tilde{u}(k) = \rev{\left(1 + \frac{0.03^2}{P^2} k^2\right)^{-a}} \Fper\tilde{u}(k) \\
			&\le C \rev{(1 + k^2)^{-a}} \Fper\tilde{u}(k)
		\end{split}
	\end{equation*}
	for some $C > 0$, so that $\ran \widetilde{T} \subseteq \rev{H^{2a}(-1,1)}$ by definition of \rev{$H^{2a}(-1,1)$} and $\widetilde{T}$ is bounded from $L^2(-1,1)$ to \rev{$H^{2a}(-1,1)$}.
\end{proof}

As the feature of interest, we aim to test whether
\[ \int_{0}^{l} \udag(x) \di x = 0 \]
for some value $l \in \left(0,1\right)$ under the a priori assumption that $\udag$ is a density. As $\udag$ is then especially non-negative, we obtain for any non-negative function $\varphi_l \in L^2(-1,1)$ with $\esssupp (\varphi_l) = \left[0,l\right]$ that
\[
\left\langle \varphi_l, \udag\right\rangle = 0 \quad \Leftrightarrow \quad \int_{0}^{l} \udag(x) \di x = 0.
\]
\rev{A natural choice for $\varphi_l$ is a symmetric $\beta$-kernel $p_l$ depending on a parameter $\beta > 0$ on the interval $[0,l]$, i.e.,
\[
\varphi_{l,\beta}(x) = p_l(x;\beta) \propto \begin{cases} x^{\beta - 1}\left(l-x\right)^{\beta-1} & \text{if } x \in [0,l], \\ 0 & \text{else}. \end{cases}.
\]
As $\overline{\ran T^*}  = L^2 \left(-1,1\right)$, we have $\varphi_{l,\beta} \in \overline{\ran T^*}$ for all $\beta > 0$. Note that for $\beta = 1$, the function $\varphi_{l,1}$ is just the indicator function of the interval $[0,l]$. However, to ensure the stronger condition $\varphi_{l,\beta} \in \ran T^*$, which allows to formally define the unregularized test $\Psi_0$, it is in view of Lemma \ref{ran_T_H4} necessary to choose $\beta \geq 1 + 2a$.
Furthermore we always normalize such that $\smallnorm{\varphi_{l,\beta}}{L^ 2(-1,1)} = 1$.}

\rev{As the truth $\udag$ we choose a scaled and shifted version of the same kernel, i.e.,
\[ u_{l,\lambda,\gamma}^\dagger(x) \propto p_l\left(x + \frac{1 - \lambda}{l};\gamma\right) \]
with $\lambda \in [0,1]$ and $\gamma > 0$, and normalize such that $\smallnorm{u_{l,\lambda,\gamma}^\dagger}{L^1(-1,1)} = 1$. This way, $u_{l,\lambda,\gamma}^\dagger$ is in fact a density supported on the interval $[(1 - \lambda)l, (2 - \lambda)l]$. 
The a priori information that $\udag$ is a density is reflected by choosing
\begin{equation*}
	\V = L^1(-1,1), \quad\text{and}\quad \V' = L^\infty(-1,1),
\end{equation*}
which directly yields $\phi_{l,\beta} \in \V'$ for all $\beta \geq 1$ and $u_{l,\lambda,\gamma}^\dagger \in \V \cap \X$ with $\smallnorm{u_{l,\lambda,\gamma}^\dagger}{\V} = 1$ for all $\lambda \in [0,1]$ and $\gamma \ge 1$. That is, \cref{ass:smooth} \ref{ass_sm_4} and \ref{ass_sm_2} are satisfied.}
The spaces $\V$ and $\V'$ are, moreover, compatible with the forward operator $\widetilde{T}$ and our choice of $\X$ in the sense of \cref{ass:smooth}.
\begin{lem}
	\label{V_compatible}
	The operator $\widetilde{T}$ satisfies \cref{ass:smooth} \ref{ass_sm_1} and \ref{ass_sm_3} with $\V = L^1(-1,1)$, $\V' = L^\infty(-1,1)$.
\end{lem}
\begin{proof}
Ad (1): It holds trivially that
\[ \scalprod{v',v}{\X^* \times \X} = \scalprod{v',v}{L^2} \le \norm{v'}{L^\infty} \norm{v}{L^1} = \norm{v'}{\V'} \norm{v}{\V} \]
for all $v' \in \V' \cap \X^*$ and $v \in \V \cap \X$.

Ad (3): First note that $T^* = T$. By the Sobolev embedding theorem, \rev{$H^{2a}(-1,1)$} is continuously embedded into $L^\infty(-1,1)$ \rev{for all $a \ge \frac12$}. Now, the proposition follows form \cref{ran_T_H4}.
\end{proof}

\rev{In total we now consider three testing scenarios, cf. also Figure \ref{fig:Scen1}:
\begin{itemize}
	\item[(S1)] The \emph{compatible smooth} scenario $\beta = 1 + 2a$ and $\gamma = 2$, where the unregularized test is formally available and $\udag$ is smooth, i.e., we test
\[
H_0^{\text{cs}} : \left\langle \varphi_{l,1 + 2a}, u_{l,\lambda,2}^\dagger\right\rangle =  0\qquad \text{vs.}\qquad H_1^{\text{cs}} : \left\langle \varphi_{l,1 + 2a}, u_{l,\lambda,2}^\dagger\right\rangle > 0.
\]
\item[(S2)] The \emph{compatible nonsmooth} scenario $\beta = 1 + 2a$ and $\gamma 1$, where the unregularized test is formally available but $\udag$ is not smooth, i.e., we test
\[
H_0^{\text{cn}} : \left\langle \varphi_{l,1 + 2a}, u_{l,\lambda,1}^\dagger\right\rangle =  0\qquad \text{vs.}\qquad H_1^{\text{cn}} : \left\langle \varphi_{l,1 + 2a}, u_{l,\lambda,1}^\dagger\right\rangle > 0.
\]
\item[(S3)] The \emph{incompatible smooth} scenario $\beta = 1$ and $\gamma = 2$, where the unregularized test is formally not available but $\udag$ is smooth, i.e., we test
\[
H_0^{\text{is}} : \left\langle \varphi_{l,1}, u_{l,\lambda,2}^\dagger\right\rangle =  0\qquad \text{vs.}\qquad H_1^{\text{is}} : \left\langle \varphi_{l,1}, u_{l,\lambda,2}^\dagger\right\rangle > 0.
\]
\end{itemize}
We consider these scenarios for different values of $a$ and $l$ (influencing the difficulty of the problem in the sense of ill-posedness) as well as different values of $\lambda$ (influencing the size $\smallscalprod{\phi,\udag}{}$ of the investigated feature).}

\begin{figure}[htp]
	\subcaptionbox*{Compatible smooth scenario (S1)}[0.33\textwidth]{\begin{tikzpicture}[baseline]
	\begin{axis}[%
		width=\fwidthscen,
		height=\fheightscen,
		scale only axis,
		xmin=-0.05,
		xmax=0.1,
		ymax=42,
		xticklabel style={/pgf/number format/.cd,fixed,precision=2},
	]
		\addplot[color of colormap=200 of viridis] table [x index=0, y index=1] {data/scenario_phi5_l20.dat};
		\label{plot:phi}
		\addplot[color of colormap=400 of viridis] table [x index=0, y index=1] {data/scenario_u2_l20_k0.dat};
		\label{plot:u1dag}
		\addplot[color of colormap=600 of viridis] table [x index=0, y index=1] {data/scenario_u2_l20_k7.dat};
		\label{plot:u2dag}
		\addplot[color of colormap=800 of viridis] table [x index=0, y index=1] {data/scenario_u2_l20_k13.dat};
		\label{plot:u3dag}
	\end{axis}
\end{tikzpicture}%
}
	\subcaptionbox*{Compatible nonsmooth scenario (S2)}[0.33\textwidth]{\begin{tikzpicture}[baseline]
	\begin{axis}[%
		width=\fwidthscen,
		height=\fheightscen,
		scale only axis,
		xmin=-0.05,
		xmax=0.1,
		ymax=42,
		xticklabel style={/pgf/number format/.cd,fixed,precision=2},
	]
		\addplot[color of colormap=200 of viridis] table [x index=0, y index=1] {data/scenario_phi5_l20.dat};
		\addplot[color of colormap=400 of viridis] table [x index=0, y index=1] {data/scenario_u1_l20_k0.dat};
		\addplot[color of colormap=600 of viridis] table [x index=0, y index=1] {data/scenario_u1_l20_k7.dat};
		\addplot[color of colormap=800 of viridis] table [x index=0, y index=1] {data/scenario_u1_l20_k13.dat};
		\addplot[color of colormap=400 of viridis, dashed] table [x index=0, y index=1] {data/scenario_u1_l20_k0_jumps.dat};
		\addplot[color of colormap=600 of viridis, dashed] table [x index=0, y index=1] {data/scenario_u1_l20_k7_jumps.dat};
		\addplot[color of colormap=800 of viridis, dashed] table [x index=0, y index=1] {data/scenario_u1_l20_k13_jumps.dat};
	\end{axis}
\end{tikzpicture}%
}
	\subcaptionbox*{Incompatible smooth scenario (S3)}[0.33\textwidth]{\begin{tikzpicture}[baseline]
	\begin{axis}[%
		width=\fwidthscen,
		height=\fheightscen,
		scale only axis,
		xmin=-0.05,
		xmax=0.1,
		ymax=42,
		xticklabel style={/pgf/number format/.cd,fixed,precision=2},
	]
		\addplot[color of colormap=200 of viridis] table [x index=0, y index=1] {data/scenario_phi1_l20.dat};
		\addplot[color of colormap=200 of viridis, dashed] table [x index=0, y index=1] {data/scenario_phi1_l20_jumps.dat};
		\addplot[color of colormap=400 of viridis] table [x index=0, y index=1] {data/scenario_u2_l20_k0.dat};
		\addplot[color of colormap=600 of viridis] table [x index=0, y index=1] {data/scenario_u2_l20_k7.dat};
		\addplot[color of colormap=800 of viridis] table [x index=0, y index=1] {data/scenario_u2_l20_k13.dat};
	\end{axis}
\end{tikzpicture}%
}
	\caption{The function $\phi_{l,\beta}$ (\ref{plot:phi}) and the truth $u_{l,\lambda,\gamma}^\dagger$ for $\lambda = 1$ (\ref{plot:u1dag}), $\lambda = \frac23$ (\ref{plot:u2dag}), and $\lambda = \frac13$ (\ref{plot:u3dag}) in case of $a = 2$ and $l = \frac{5}{128}$.}
	\label{fig:Scen1}
\end{figure}
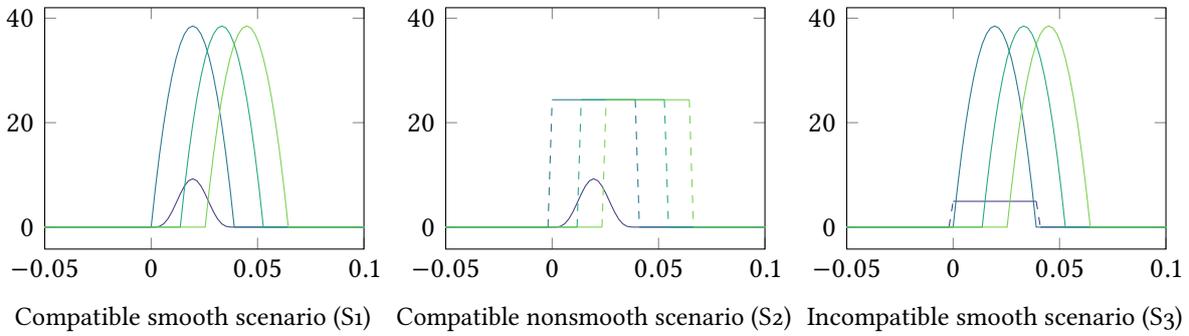

For our simulations we use Gaussian white noise. For $t > \frac{d}{2} = \frac12$, the Gaussian white noise process $Z$ almost surely takes values in $H^{-t}(-1,1) = H^t(-1,1)^*$, i.e., $Z$ is almost surely a bounded linear functional on $H^t(-1,1)$. For sharper results, see \cite{veraar2011}. In the following, we therefore choose
\begin{equation*}
	\calZ = H^t(-1,1)
\end{equation*}
for different values of $t \ge 0$.

Throughout what follows, all tests are constructed to have a level of at most $\alpha = 0.1$.
We compute the exact power of the unregularized test $\Psi_{\Phi_0,c_0}$ and the oracle test $\Psi_{\Phi^\dagger,c^\dagger}$ based upon expressions \eqref{eq:power_unreg} and \eqref{eq:power_oracle}.
\rev{For more details, see \cref{sec:norm_disc}.}
For the adaptive test $\Psi^*(\cdot;Y_1)$, we compute the empirical power
\begin{align*}
	\P{1}{\Psi^* = 1} &= \E{\Psi^*(Y_2,Y_1)} = \E{\E{\Psi^*(Y_2,Y_1)\vert Y_1}} \\
	&\approx \frac{1}{M} \sum_{m = 1}^M \P{1}{\Psi^*(Y_2,y_1^m) = 1\vert Y_1 = y_1^M} \\
	&= \frac1M\sum_{\genfrac{}{}{0pt}{}{m=1}{\min J_{y_1^m}^{\calZ}~\text{exists}}}^M Q\left(q_\alpha^\Normal - \frac{J_{T\udag}^\Y(\Phi(y_1^m))}{\sigma}\right)
\end{align*}
using $M = 100$ independent samples $\{y_1^1, \dots, y_1^M\}$ of the data. 
\rev{The probe element of the adaptive test is found as a solution of the convex surrogate problem \eqref{eq:approx_conv_prob} discussed in \cref{sec:computability}, where the minimizer is computed numerically using a primal-dual proximal splitting method. For more details on the implementation of this method, see \cref{sec:impl_min}.}
Here and in what follows, existence of $\min J_{y_1^m}^{\calZ}$ is numerically understood as convergence of the discussed minimization algorithm to \rev{an element} $\left(e,s\right)$ with $s \neq 0$.

Since the adaptive test uses two samples of the data compared to one in case of the unregularized and oracle test, we treat the latter two as if they, too, had access to two samples of the data by reducing the noise level of their data by a factor of $\sqrt{2}$.

\subsection{Numerical results}

\subsubsection{Compatible smooth scenario (S1)}

The results for the compatible smooth scenario (S1) are depicted in Figure \ref{fig:compat}.
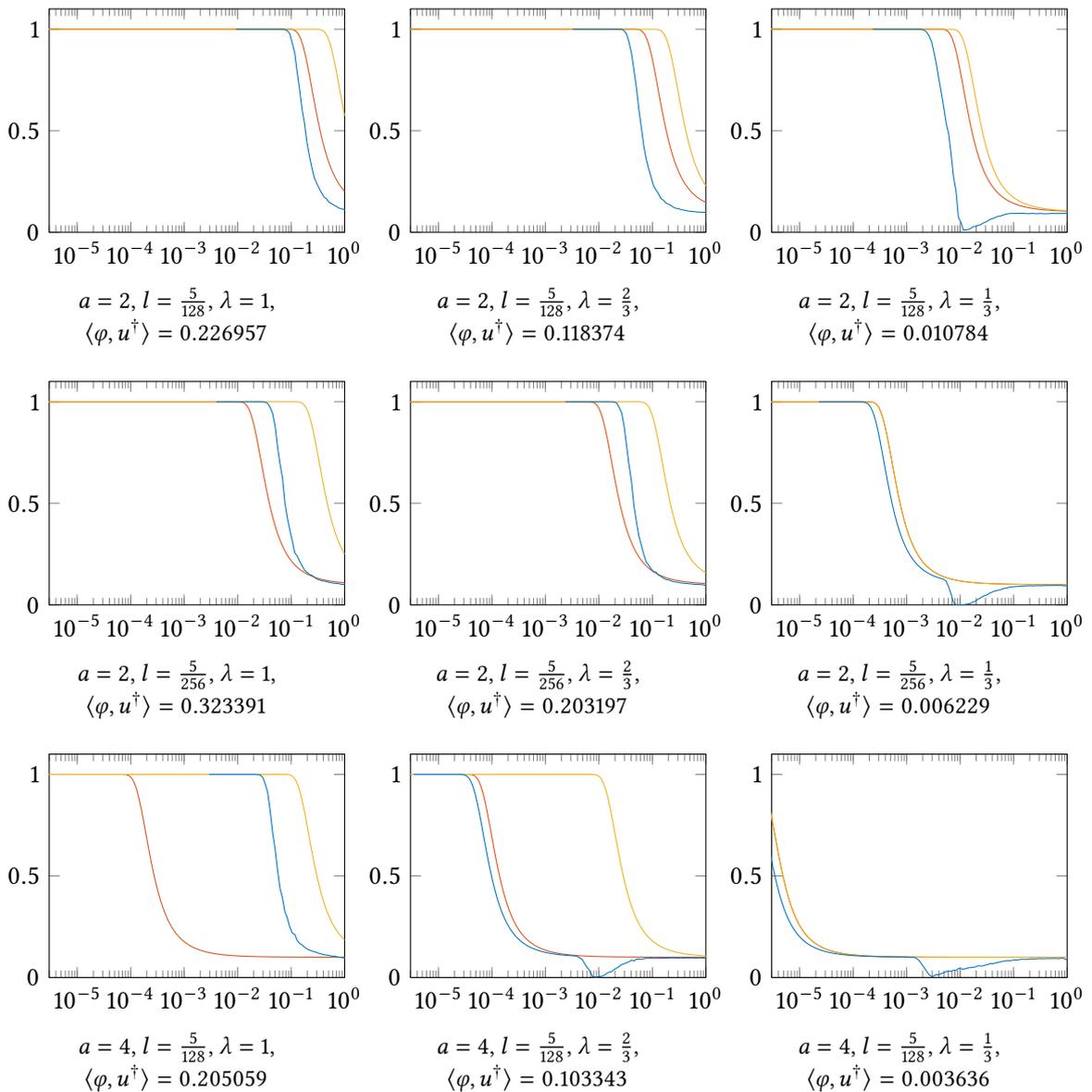
\begin{figure}[!htb]
	\centering
	\subcaptionbox*{$a = 2$, $l = \frac{5}{128}$, $\lambda = 1$,\\$\smallscalprod{\phi, \udag}{} = 0.226957$}[0.32\textwidth][l]{\begin{tikzpicture}[baseline]

\begin{axis}[%
	width=\fwidthcompat,
	height=\fheightcompat,
	scale only axis,
	xmode=log,
	xmin=\sigmamincompat,
	xmax=1,
	xminorticks=true,
	ymin=0,
	ymax=1.1,
]

\addplot[color=mycolor2] table [x index=0, y index=1] {data/power_prob1_a2_t51_phi5_u2_l20_k0_unreg.dat}; 
\addplot[color=mycolor3] table [x index=0, y index=2] {data/power_prob1_a2_t51_phi5_u2_l20_k0_unreg.dat}; 
\addplot[color=mycolor1] table [x index=0, y index=1] {data/power_prob1_a2_t51_phi5_u2_l20_k0_M100.dat}; 

\end{axis}

\end{tikzpicture}%
}
	\subcaptionbox*{$a = 2$, $l = \frac{5}{128}$, $\lambda = \frac23$,\\$\smallscalprod{\phi, \udag}{} = 0.118374$}[0.32\textwidth][l]{\begin{tikzpicture}[baseline]

\begin{axis}[%
	width=\fwidthcompat,
	height=\fheightcompat,
	scale only axis,
	xmode=log,
	xmin=\sigmamincompat,
	xmax=1,
	xminorticks=true,
	ymin=0,
	ymax=1.1,
]

\addplot[color=mycolor2] table [x index=0, y index=1] {data/power_prob1_a2_t51_phi5_u2_l20_k7_unreg.dat}; 
\addplot[color=mycolor3] table [x index=0, y index=2] {data/power_prob1_a2_t51_phi5_u2_l20_k7_unreg.dat}; 
\addplot[color=mycolor1] table [x index=0, y index=1] {data/power_prob1_a2_t51_phi5_u2_l20_k7_M100.dat}; 

\end{axis}

\end{tikzpicture}%
}
	\subcaptionbox*{$a = 2$, $l = \frac{5}{128}$, $\lambda = \frac13$,\\$\smallscalprod{\phi, \udag}{} = 0.010784$}[0.32\textwidth][l]{\begin{tikzpicture}[baseline]

\begin{axis}[%
	width=\fwidthcompat,
	height=\fheightcompat,
	scale only axis,
	xmode=log,
	xmin=\sigmamincompat,
	xmax=1,
	xminorticks=true,
	ymin=0,
	ymax=1.1,
]

\addplot[color=mycolor2] table [x index=0, y index=1] {data/power_prob1_a2_t51_phi5_u2_l20_k13_unreg.dat}; 
\addplot[color=mycolor3] table [x index=0, y index=2] {data/power_prob1_a2_t51_phi5_u2_l20_k13_unreg.dat}; 
\addplot[color=mycolor1] table [x index=0, y index=1] {data/power_prob1_a2_t51_phi5_u2_l20_k13_M100.dat}; 

\end{axis}

\end{tikzpicture}%
}\\
	\bigskip
	\subcaptionbox*{$a = 2$, $l = \frac{5}{256}$, $\lambda = 1$,\\$\smallscalprod{\phi, \udag}{} = 0.323391$}[0.32\textwidth][l]{\begin{tikzpicture}[baseline]

\begin{axis}[%
	width=\fwidthcompat,
	height=\fheightcompat,
	scale only axis,
	xmode=log,
	xmin=\sigmamincompat,
	xmax=1,
	xminorticks=true,
	ymin=0,
	ymax=1.1,
]

\addplot[color=mycolor2] table [x index=0, y index=1] {data/power_prob1_a2_t51_phi5_u2_l10_k0_unreg.dat}; 
\addplot[color=mycolor3] table [x index=0, y index=2] {data/power_prob1_a2_t51_phi5_u2_l10_k0_unreg.dat}; 
\addplot[color=mycolor1] table [x index=0, y index=1] {data/power_prob1_a2_t51_phi5_u2_l10_k0_M100.dat}; 

\end{axis}

\end{tikzpicture}%
}
	\subcaptionbox*{$a = 2$, $l = \frac{5}{256}$, $\lambda = \frac23$,\\$\smallscalprod{\phi, \udag}{} = 0.203197$}[0.32\textwidth][l]{\begin{tikzpicture}[baseline]

\begin{axis}[%
	width=\fwidthcompat,
	height=\fheightcompat,
	scale only axis,
	xmode=log,
	xmin=\sigmamincompat,
	xmax=1,
	xminorticks=true,
	ymin=0,
	ymax=1.1,
]

\addplot[color=mycolor2] table [x index=0, y index=1] {data/power_prob1_a2_t51_phi5_u2_l10_k3_unreg.dat}; 
\addplot[color=mycolor3] table [x index=0, y index=2] {data/power_prob1_a2_t51_phi5_u2_l10_k3_unreg.dat}; 
\addplot[color=mycolor1] table [x index=0, y index=1] {data/power_prob1_a2_t51_phi5_u2_l10_k3_M100.dat}; 

\end{axis}

\end{tikzpicture}%
}
	\subcaptionbox*{$a = 2$, $l = \frac{5}{256}$, $\lambda = \frac13$,\\$\smallscalprod{\phi, \udag}{} = 0.006229$}[0.32\textwidth][l]{\begin{tikzpicture}[baseline]

\begin{axis}[%
	width=\fwidthcompat,
	height=\fheightcompat,
	scale only axis,
	xmode=log,
	xmin=\sigmamincompat,
	xmax=1,
	xminorticks=true,
	ymin=0,
	ymax=1.1,
]

\addplot[color=mycolor2] table [x index=0, y index=1] {data/power_prob1_a2_t51_phi5_u2_l10_k7_unreg.dat}; 
\addplot[color=mycolor3] table [x index=0, y index=2] {data/power_prob1_a2_t51_phi5_u2_l10_k7_unreg.dat}; 
\addplot[color=mycolor1] table [x index=0, y index=1] {data/power_prob1_a2_t51_phi5_u2_l10_k7_M100.dat}; 

\end{axis}

\end{tikzpicture}%
}\\
	\bigskip
	\subcaptionbox*{$a = 4$, $l = \frac{5}{128}$, $\lambda = 1$,\\$\smallscalprod{\phi, \udag}{} = 0.205059$}[0.32\textwidth][l]{\begin{tikzpicture}[baseline]

\begin{axis}[%
	width=\fwidthcompat,
	height=\fheightcompat,
	scale only axis,
	xmode=log,
	xmin=\sigmamincompat,
	xmax=1,
	xminorticks=true,
	ymin=0,
	ymax=1.1,
]

\addplot[color=mycolor2] table [x index=0, y index=1] {data/power_prob1_a4_t51_phi9_u2_l20_k0_unreg.dat}; 
\addplot[color=mycolor3] table [x index=0, y index=2] {data/power_prob1_a4_t51_phi9_u2_l20_k0_unreg.dat}; 
\addplot[color=mycolor1] table [x index=0, y index=1] {data/power_prob1_a4_t51_phi9_u2_l20_k0_M100.dat}; 

\end{axis}

\end{tikzpicture}%
}
	\subcaptionbox*{$a = 4$, $l = \frac{5}{128}$, $\lambda = \frac23$,\\$\smallscalprod{\phi, \udag}{} = 0.103343$}[0.32\textwidth][l]{\begin{tikzpicture}[baseline]

\begin{axis}[%
	width=\fwidthcompat,
	height=\fheightcompat,
	scale only axis,
	xmode=log,
	xmin=\sigmamincompat,
	xmax=1,
	xminorticks=true,
	ymin=0,
	ymax=1.1,
]

\addplot[color=mycolor2] table [x index=0, y index=1] {data/power_prob1_a4_t51_phi9_u2_l20_k7_unreg.dat}; 
\addplot[color=mycolor3] table [x index=0, y index=2] {data/power_prob1_a4_t51_phi9_u2_l20_k7_unreg.dat}; 
\addplot[color=mycolor1] table [x index=0, y index=1] {data/power_prob1_a4_t51_phi9_u2_l20_k7_M100.dat}; 

\end{axis}

\end{tikzpicture}%
}
	\subcaptionbox*{$a = 4$, $l = \frac{5}{128}$, $\lambda = \frac13$,\\ $\smallscalprod{\phi, \udag}{} = 0.003636$}[0.32\textwidth][l]{\begin{tikzpicture}[baseline]

\begin{axis}[%
	width=\fwidthcompat,
	height=\fheightcompat,
	scale only axis,
	xmode=log,
	xmin=\sigmamincompat,
	xmax=1,
	xminorticks=true,
	ymin=0,
	ymax=1.1,
]

\addplot[color=mycolor2] table [x index=0, y index=1] {data/power_prob1_a4_t51_phi9_u2_l20_k13_unreg.dat}; 
\addplot[color=mycolor3] table [x index=0, y index=2] {data/power_prob1_a4_t51_phi9_u2_l20_k13_unreg.dat}; 
\addplot[color=mycolor1] table [x index=0, y index=1] {data/power_prob1_a4_t51_phi9_u2_l20_k13_M100.dat}; 

\end{axis}

\end{tikzpicture}%
}
	\caption{Exact powers of the unregularized test (\ref{plot:unreg}), the oracle test (\ref{plot:oracle}), and empirical power of the adaptive test for $\calZ = H^{0.51}$ (\ref{plot:adaptive_Ht}) from $100$ samples against the noise level $\sigma$ in the compatible smooth scenario (S1) \rev{for $a \in \{2,4\}$, $l \in \{\frac{5}{128},\frac{5}{256}\}$, and $\lambda \in \left\{\frac13,\frac23,1\right\}$. The top row shows the results for $a = 2$ and $l = \frac{5}{128}$, the middle row the results for $a = 2$ and $l = \frac{5}{256}$, and the bottom row the results for $a = 2$ and $l = \frac{5}{128}$.}}
	\label{fig:compat}
\end{figure}
We find that the optimal test $\Psi_{\Phi^\dagger, c^\dagger}$ is --- in agreement with the theory --- superior over the unregularized test $\Psi_0$ and as to be expected also over the adaptive test $\Psi^*$. This shows that regularized hypothesis testing in fact resolves the issue (I2) raised in the introduction. It might, however, seem surprising that the unregularized test $\Psi_0$ sometimes shows a better power than the adaptive test $\Psi^*$. 
\rev{We find that this is only the case for milder ill-posed problems, i.e., for $a = 2$ and $l = \frac{5}{128}$.
For smaller values of $l$,} all three tests lose some power due to the smaller support size, but for the unregularized test $\Psi_0$, this loss is by far larger than for the other tests. Especially, the adaptive test $\Psi^*$ has now in all relevant situations a larger power than $\Psi_0$.
In the case $\lambda = \frac13$, none of the three tests achieves a significantly nontrivial power within the considerd range of the noise level. This is caused by a particularly small feature size $\langle \phi_{5/256,5}, u_{5/256,1/3,2}^\dagger \rangle$ in this case.
\rev{The power plots for $a = 4$ illustrate that a higher order of ill-posedness causes a higher difficulty of the problem, but the oracle and the adaptive test are in principle able to cope with this as visible in case $\lambda = 1$. However, if the support overlap $\lambda$ and with it also the feature size $\smallscalprod{\phi,\udag}{}$ becomes smaller, first the adaptive test and finally also the oracle test fall back to the unregularized test.}

\subsubsection{Compatible nonsmooth scenario (S2)}

The results for the compatible nonsmooth scenario (S2) are depicted in Figure \ref{fig:compat_nonsmooth}.
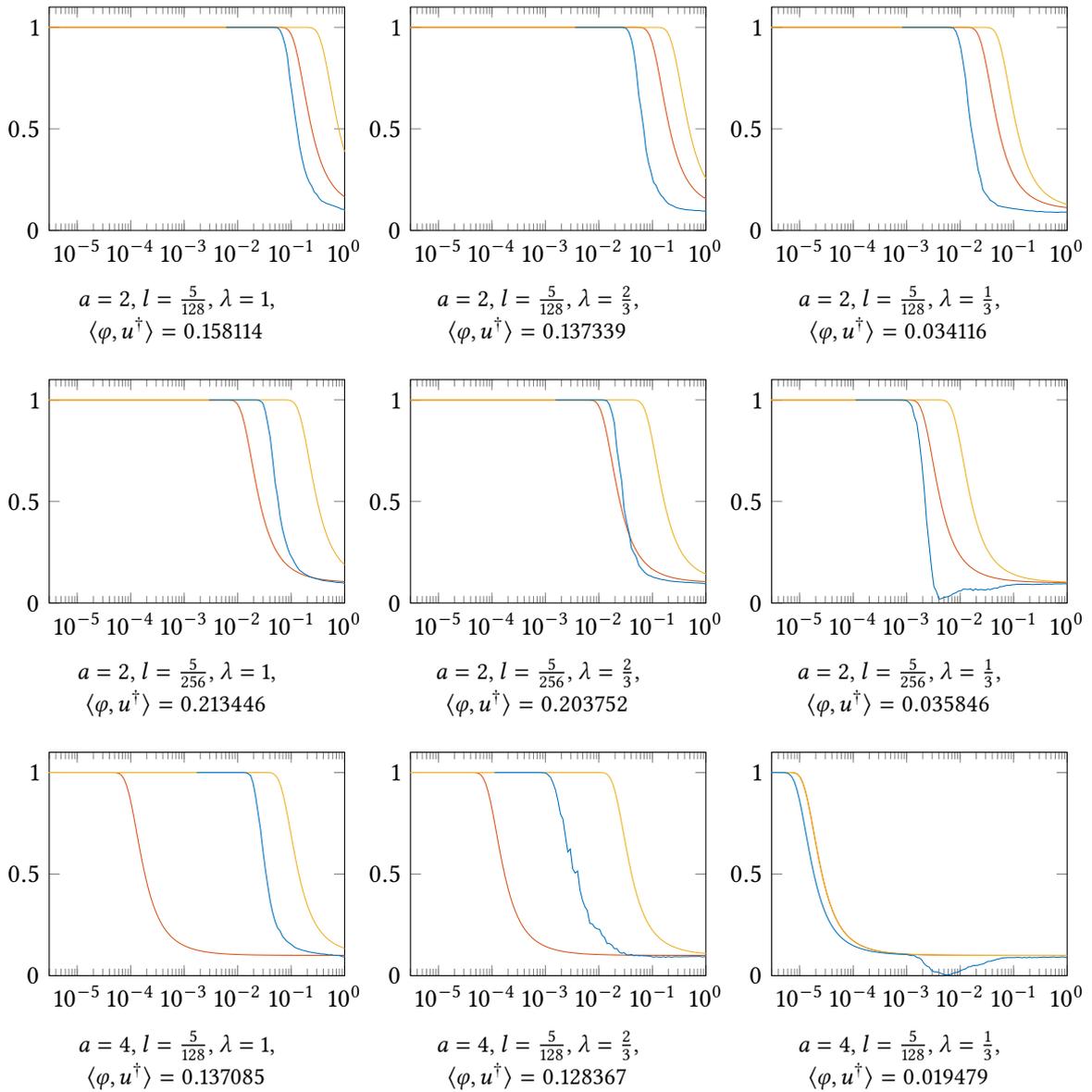
\begin{figure}[!htb]
	\centering
	\subcaptionbox*{$a = 2$, $l = \frac{5}{128}$, $\lambda = 1$,\\$\smallscalprod{\phi, \udag}{} = 0.158114$}[0.32\textwidth][l]{\begin{tikzpicture}[baseline]

\begin{axis}[%
	width=\fwidthcompat,
	height=\fheightcompat,
	scale only axis,
	xmode=log,
	xmin=\sigmamincompat,
	xmax=1,
	xminorticks=true,
	ymin=0,
	ymax=1.1,
]

\addplot[color=mycolor2] table [x index=0, y index=1] {data/power_prob1_a2_t51_phi5_u1_l20_k0_unreg.dat}; 
\addplot[color=mycolor3] table [x index=0, y index=2] {data/power_prob1_a2_t51_phi5_u1_l20_k0_unreg.dat}; 
\addplot[color=mycolor1] table [x index=0, y index=1] {data/power_prob1_a2_t51_phi5_u1_l20_k0_M100.dat}; 

\end{axis}

\end{tikzpicture}%
}
	\subcaptionbox*{$a = 2$, $l = \frac{5}{128}$, $\lambda = \frac23$,\\$\smallscalprod{\phi, \udag}{} = 0.137339$}[0.32\textwidth][l]{\begin{tikzpicture}[baseline]

\begin{axis}[%
	width=\fwidthcompat,
	height=\fheightcompat,
	scale only axis,
	xmode=log,
	xmin=\sigmamincompat,
	xmax=1,
	xminorticks=true,
	ymin=0,
	ymax=1.1,
]

\addplot[color=mycolor2] table [x index=0, y index=1] {data/power_prob1_a2_t51_phi5_u1_l20_k7_unreg.dat}; 
\addplot[color=mycolor3] table [x index=0, y index=2] {data/power_prob1_a2_t51_phi5_u1_l20_k7_unreg.dat}; 
\addplot[color=mycolor1] table [x index=0, y index=1] {data/power_prob1_a2_t51_phi5_u1_l20_k7_M100.dat}; 

\end{axis}

\end{tikzpicture}%
}
	\subcaptionbox*{$a = 2$, $l = \frac{5}{128}$, $\lambda = \frac13$,\\ $\smallscalprod{\phi, \udag}{} = 0.034116$}[0.32\textwidth][l]{\begin{tikzpicture}[baseline]

\begin{axis}[%
	width=\fwidthcompat,
	height=\fheightcompat,
	scale only axis,
	xmode=log,
	xmin=\sigmamincompat,
	xmax=1,
	xminorticks=true,
	ymin=0,
	ymax=1.1,
]

\addplot[color=mycolor2] table [x index=0, y index=1] {data/power_prob1_a2_t51_phi5_u1_l20_k13_unreg.dat}; 
\addplot[color=mycolor3] table [x index=0, y index=2] {data/power_prob1_a2_t51_phi5_u1_l20_k13_unreg.dat}; 
\addplot[color=mycolor1] table [x index=0, y index=1] {data/power_prob1_a2_t51_phi5_u1_l20_k13_M100.dat}; 

\end{axis}

\end{tikzpicture}%
}\\
	\bigskip
	\subcaptionbox*{$a = 2$, $l = \frac{5}{256}$, $\lambda = 1$,\\$\smallscalprod{\phi, \udag}{} = 0.213446$}[0.32\textwidth][l]{\begin{tikzpicture}[baseline]

\begin{axis}[%
	width=\fwidthcompat,
	height=\fheightcompat,
	scale only axis,
	xmode=log,
	xmin=\sigmamincompat,
	xmax=1,
	xminorticks=true,
	ymin=0,
	ymax=1.1,
]

\addplot[color=mycolor2] table [x index=0, y index=1] {data/power_prob1_a2_t51_phi5_u1_l10_k0_unreg.dat}; 
\addplot[color=mycolor3] table [x index=0, y index=2] {data/power_prob1_a2_t51_phi5_u1_l10_k0_unreg.dat}; 
\addplot[color=mycolor1] table [x index=0, y index=1] {data/power_prob1_a2_t51_phi5_u1_l10_k0_M100.dat}; 

\end{axis}

\end{tikzpicture}%
}
	\subcaptionbox*{$a = 2$, $l = \frac{5}{256}$, $\lambda = \frac23$,\\$\smallscalprod{\phi, \udag}{} = 0.203752$}[0.32\textwidth][l]{\begin{tikzpicture}[baseline]

\begin{axis}[%
	width=\fwidthcompat,
	height=\fheightcompat,
	scale only axis,
	xmode=log,
	xmin=\sigmamincompat,
	xmax=1,
	xminorticks=true,
	ymin=0,
	ymax=1.1,
]

\addplot[color=mycolor2] table [x index=0, y index=1] {data/power_prob1_a2_t51_phi5_u1_l10_k3_unreg.dat}; 
\addplot[color=mycolor3] table [x index=0, y index=2] {data/power_prob1_a2_t51_phi5_u1_l10_k3_unreg.dat}; 
\addplot[color=mycolor1] table [x index=0, y index=1] {data/power_prob1_a2_t51_phi5_u1_l10_k3_M100.dat}; 

\end{axis}

\end{tikzpicture}%
}
	\subcaptionbox*{$a = 2$, $l = \frac{5}{256}$, $\lambda = \frac13$,\\$\smallscalprod{\phi, \udag}{} = 0.035846$}[0.32\textwidth][l]{\begin{tikzpicture}[baseline]

\begin{axis}[%
	width=\fwidthcompat,
	height=\fheightcompat,
	scale only axis,
	xmode=log,
	xmin=\sigmamincompat,
	xmax=1,
	xminorticks=true,
	ymin=0,
	ymax=1.1,
]

\addplot[color=mycolor2] table [x index=0, y index=1] {data/power_prob1_a2_t51_phi5_u1_l10_k7_unreg.dat}; 
\addplot[color=mycolor3] table [x index=0, y index=2] {data/power_prob1_a2_t51_phi5_u1_l10_k7_unreg.dat}; 
\addplot[color=mycolor1] table [x index=0, y index=1] {data/power_prob1_a2_t51_phi5_u1_l10_k7_M100.dat}; 

\end{axis}

\end{tikzpicture}%
}\\
	\bigskip
	\subcaptionbox*{$a = 4$, $l = \frac{5}{128}$, $\lambda = 1$, $\smallscalprod{\phi, \udag}{} = 0.137085$}[0.32\textwidth][l]{\begin{tikzpicture}[baseline]

\begin{axis}[%
	width=\fwidthcompat,
	height=\fheightcompat,
	scale only axis,
	xmode=log,
	xmin=\sigmamincompat,
	xmax=1,
	xminorticks=true,
	ymin=0,
	ymax=1.1,
]

\addplot[color=mycolor2] table [x index=0, y index=1] {data/power_prob1_a4_t51_phi9_u1_l20_k0_unreg.dat}; 
\addplot[color=mycolor3] table [x index=0, y index=2] {data/power_prob1_a4_t51_phi9_u1_l20_k0_unreg.dat}; 
\addplot[color=mycolor1] table [x index=0, y index=1] {data/power_prob1_a4_t51_phi9_u1_l20_k0_M100.dat}; 

\end{axis}

\end{tikzpicture}%
}
	\subcaptionbox*{$a = 4$, $l = \frac{5}{128}$, $\lambda = \frac23$,\\$\smallscalprod{\phi, \udag}{} = 0.128367$}[0.32\textwidth][l]{\begin{tikzpicture}[baseline]

\begin{axis}[%
	width=\fwidthcompat,
	height=\fheightcompat,
	scale only axis,
	xmode=log,
	xmin=\sigmamincompat,
	xmax=1,
	xminorticks=true,
	ymin=0,
	ymax=1.1,
]

\addplot[color=mycolor2] table [x index=0, y index=1] {data/power_prob1_a4_t51_phi9_u1_l20_k7_unreg.dat}; 
\addplot[color=mycolor3] table [x index=0, y index=2] {data/power_prob1_a4_t51_phi9_u1_l20_k7_unreg.dat}; 
\addplot[color=mycolor1] table [x index=0, y index=1] {data/power_prob1_a4_t51_phi9_u1_l20_k7_M100.dat}; 

\end{axis}

\end{tikzpicture}%
}
	\subcaptionbox*{$a = 4$, $l = \frac{5}{128}$, $\lambda = \frac13$,\\$\smallscalprod{\phi, \udag}{} = 0.019479$}[0.32\textwidth][l]{\begin{tikzpicture}[baseline]

\begin{axis}[%
	width=\fwidthcompat,
	height=\fheightcompat,
	scale only axis,
	xmode=log,
	xmin=\sigmamincompat,
	xmax=1,
	xminorticks=true,
	ymin=0,
	ymax=1.1,
]

\addplot[color=mycolor2] table [x index=0, y index=1] {data/power_prob1_a4_t51_phi9_u1_l20_k13_unreg.dat}; 
\addplot[color=mycolor3] table [x index=0, y index=2] {data/power_prob1_a4_t51_phi9_u1_l20_k13_unreg.dat}; 
\addplot[color=mycolor1] table [x index=0, y index=1] {data/power_prob1_a4_t51_phi9_u1_l20_k13_M100.dat}; 

\end{axis}

\end{tikzpicture}%
}
	\caption{Exact powers of the unregularized test (\ref{plot:unreg}), the oracle test (\ref{plot:oracle}), and empirical power of the adaptive test for $\calZ = H^{0.51}$ (\ref{plot:adaptive_Ht}) from $100$ samples against the noise level $\sigma$ in the compatible nonsmooth scenario (S2) \rev{for $a \in \{2,4\}$, $l \in \{\frac{5}{128},\frac{5}{256}\}$, and $\lambda \in \left\{\frac13,\frac23,1\right\}$. The top row shows the results for $a = 2$ and $l = \frac{5}{128}$, the middle row the results for $a = 2$ and $l = \frac{5}{256}$, and the bottom row the results for $a = 2$ and $l = \frac{5}{128}$.}}
	\label{fig:compat_nonsmooth}
\end{figure}
The findings are similar to those of the compatible smooth scenario. The optimal test $\Psi_{\Phi^\dagger, c^\dagger}$ is always superior compared to the unregularized test, which itself is slightly better than the adaptive test $\Psi^*$.
\rev{However, the decreased smoothness of $\udag$ and, consequently, the increased feature size $\smallscalprod{\phi,\udag}{}$ lead to an improved power of the oracle and the adaptive test compared to scenario (S1).}

\subsubsection{Incompatible smooth scenario (S3)}

The results for the incompatible smooth scenario (S3) are depicted in Figure \ref{fig:incomp_smooth}. Compared to the previous situations, we also investigate the choice of the smoothness parameter $t$ in the space $\calZ = H^t \left(-1,1\right)$. Recall, that our theoretical results are valid only for $t > \frac12$, and note that the unregularized test is also shown for comparison where the equation $T^*\Phi_0 = \varphi_l^{(1,1)}$ is solved numerically despite the formal non-existence of a solution in $L^2(-1,1)$.
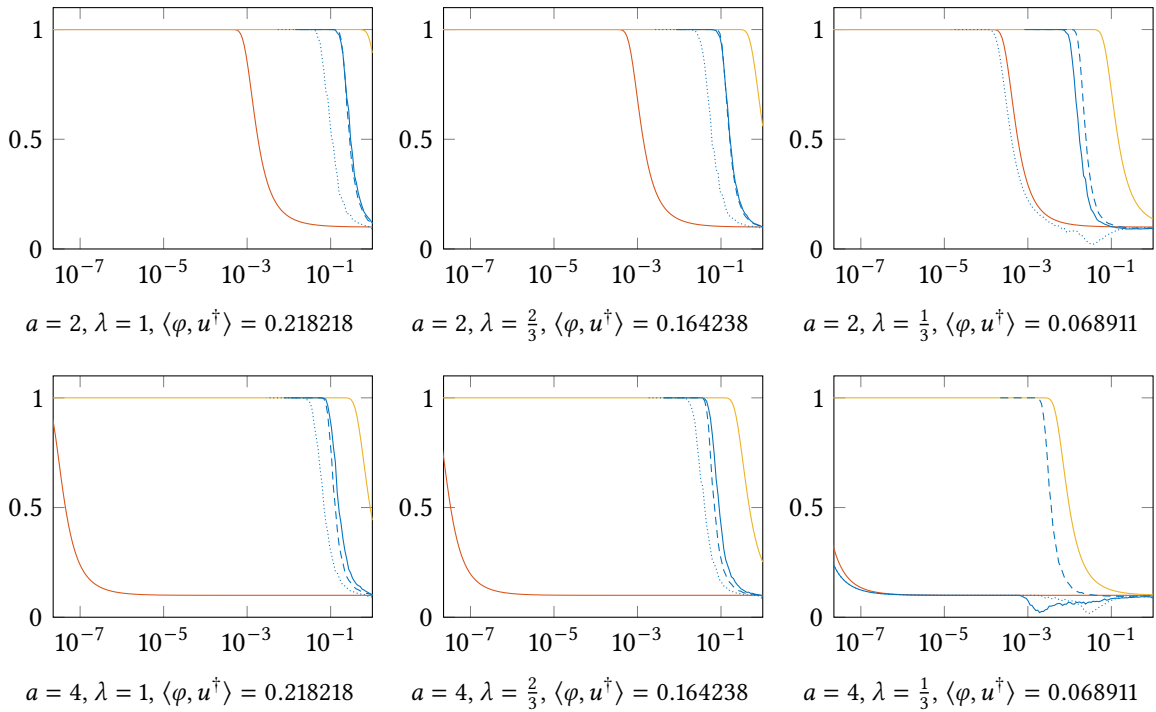
\begin{figure}[!htb]
	\centering
	\subcaptionbox*{$a = 2$, $\lambda = 1$, $\smallscalprod{\phi, \udag}{} = 0.218218$}[0.32\textwidth][l]{\begin{tikzpicture}[baseline]

\begin{axis}[%
	width=\fwidthincompat,
	height=\fheightincompat,
	scale only axis,
	xmode=log,
	xmin=\sigmaminincompat,
	xmax=1,
	xminorticks=true,
	ymin=0,
	ymax=1.1,
]
\addplot[color=mycolor2] table [x index=0, y index=1] {data/power_prob1_a2_t51_phi1_u2_l20_k0_unreg.dat}; 
\label{plot:unreg}
\addplot[color=mycolor3] table [x index=0, y index=2] {data/power_prob1_a2_t51_phi1_u2_l20_k0_unreg.dat};
\label{plot:oracle}
\addplot[color=mycolor1,densely dashed] table [x index=0, y index=1] {data/power_prob1_a2_t0_phi1_u2_l20_k0_M100.dat}; 
\label{plot:adaptive_L2}
\addplot[color=mycolor1] table [x index=0, y index=1] {data/power_prob1_a2_t51_phi1_u2_l20_k0_M100.dat};
\label{plot:adaptive_Ht}
\addplot[color=mycolor1,densely dotted] table [x index=0, y index=1] {data/power_prob1_a2_t100_phi1_u2_l20_k0_M100.dat};
\label{plot:adaptive_H1}

\end{axis}

\end{tikzpicture}%
}
	\subcaptionbox*{$a = 2$, $\lambda = \frac23$, $\smallscalprod{\phi, \udag}{} = 0.164238$}[0.32\textwidth][l]{\begin{tikzpicture}[baseline]

\begin{axis}[%
	width=\fwidthincompat,
	height=\fheightincompat,
	scale only axis,
	xmode=log,
	xmin=\sigmaminincompat,
	xmax=1,
	xminorticks=true,
	ymin=0,
	ymax=1.1,
]

\addplot[color=mycolor2] table [x index=0, y index=1] {data/power_prob1_a2_t51_phi1_u2_l20_k7_unreg.dat}; 
\addplot[color=mycolor3] table [x index=0, y index=2] {data/power_prob1_a2_t51_phi1_u2_l20_k7_unreg.dat};
\addplot[color=mycolor1,densely dashed] table [x index=0, y index=1] {data/power_prob1_a2_t0_phi1_u2_l20_k7_M100.dat}; 
\addplot[color=mycolor1] table [x index=0, y index=1] {data/power_prob1_a2_t51_phi1_u2_l20_k7_M100.dat};
\addplot[color=mycolor1,densely dotted] table [x index=0, y index=1] {data/power_prob1_a2_t100_phi1_u2_l20_k7_M100.dat};

\end{axis}

\end{tikzpicture}%
}
	\subcaptionbox*{$a = 2$, $\lambda = \frac13$, $\smallscalprod{\phi, \udag}{} = 0.068911$}[0.32\textwidth][l]{\begin{tikzpicture}[baseline]

\begin{axis}[%
	width=\fwidthincompat,
	height=\fheightincompat,
	scale only axis,
	xmode=log,
	xmin=\sigmaminincompat,
	xmax=1,
	xminorticks=true,
	ymin=0,
	ymax=1.1,
]

\addplot[color=mycolor2] table [x index=0, y index=1] {data/power_prob1_a2_t51_phi1_u2_l20_k13_unreg.dat}; 
\addplot[color=mycolor3] table [x index=0, y index=2] {data/power_prob1_a2_t51_phi1_u2_l20_k13_unreg.dat};
\addplot[color=mycolor1,densely dashed] table [x index=0, y index=1] {data/power_prob1_a2_t0_phi1_u2_l20_k13_M100.dat}; 
\addplot[color=mycolor1] table [x index=0, y index=1] {data/power_prob1_a2_t51_phi1_u2_l20_k13_M100.dat};
\addplot[color=mycolor1,densely dotted] table [x index=0, y index=1] {data/power_prob1_a2_t100_phi1_u2_l20_k13_M100.dat};

\end{axis}

\end{tikzpicture}%
}\\
	\bigskip
	\subcaptionbox*{$a = 4$, $\lambda = 1$, $\smallscalprod{\phi, \udag}{} = 0.218218$}[0.32\textwidth][l]{\begin{tikzpicture}[baseline]

\begin{axis}[%
	width=\fwidthincompat,
	height=\fheightincompat,
	scale only axis,
	xmode=log,
	xmin=\sigmaminincompat,
	xmax=1,
	xminorticks=true,
	ymin=0,
	ymax=1.1,
]
\addplot[color=mycolor2] table [x index=0, y index=1] {data/power_prob1_a4_t51_phi1_u2_l20_k0_unreg.dat}; 
\addplot[color=mycolor3] table [x index=0, y index=2] {data/power_prob1_a4_t51_phi1_u2_l20_k0_unreg.dat};
\addplot[color=mycolor1,densely dashed] table [x index=0, y index=1] {data/power_prob1_a4_t0_phi1_u2_l20_k0_M100.dat}; 
\addplot[color=mycolor1] table [x index=0, y index=1] {data/power_prob1_a4_t51_phi1_u2_l20_k0_M100.dat};
\addplot[color=mycolor1,densely dotted] table [x index=0, y index=1] {data/power_prob1_a4_t100_phi1_u2_l20_k0_M100.dat};

\end{axis}

\end{tikzpicture}%
}
	\subcaptionbox*{$a = 4$, $\lambda = \frac23$, $\smallscalprod{\phi, \udag}{} = 0.164238$}[0.32\textwidth][l]{\begin{tikzpicture}[baseline]

\begin{axis}[%
	width=\fwidthincompat,
	height=\fheightincompat,
	scale only axis,
	xmode=log,
	xmin=\sigmaminincompat,
	xmax=1,
	xminorticks=true,
	ymin=0,
	ymax=1.1,
]

\addplot[color=mycolor2] table [x index=0, y index=1] {data/power_prob1_a4_t51_phi1_u2_l20_k7_unreg.dat}; 
\addplot[color=mycolor3] table [x index=0, y index=2] {data/power_prob1_a4_t51_phi1_u2_l20_k7_unreg.dat};
\addplot[color=mycolor1,densely dashed] table [x index=0, y index=1] {data/power_prob1_a4_t0_phi1_u2_l20_k7_M100.dat}; 
\addplot[color=mycolor1] table [x index=0, y index=1] {data/power_prob1_a4_t51_phi1_u2_l20_k7_M100.dat};
\addplot[color=mycolor1,densely dotted] table [x index=0, y index=1] {data/power_prob1_a4_t100_phi1_u2_l20_k7_M100.dat};

\end{axis}

\end{tikzpicture}%
}
	\subcaptionbox*{$a = 4$, $\lambda = \frac13$, $\smallscalprod{\phi, \udag}{} = 0.068911$}[0.32\textwidth][l]{\begin{tikzpicture}[baseline]

\begin{axis}[%
	width=\fwidthincompat,
	height=\fheightincompat,
	scale only axis,
	xmode=log,
	xmin=\sigmaminincompat,
	xmax=1,
	xminorticks=true,
	ymin=0,
	ymax=1.1,
]

\addplot[color=mycolor2] table [x index=0, y index=1] {data/power_prob1_a4_t51_phi1_u2_l20_k13_unreg.dat}; 
\addplot[color=mycolor3] table [x index=0, y index=2] {data/power_prob1_a4_t51_phi1_u2_l20_k13_unreg.dat};
\addplot[color=mycolor1,densely dashed] table [x index=0, y index=1] {data/power_prob1_a4_t0_phi1_u2_l20_k13_M100.dat}; 
\addplot[color=mycolor1] table [x index=0, y index=1] {data/power_prob1_a4_t51_phi1_u2_l20_k13_M100.dat};
\addplot[color=mycolor1,densely dotted] table [x index=0, y index=1] {data/power_prob1_a4_t100_phi1_u2_l20_k13_M100.dat};

\end{axis}

\end{tikzpicture}%
}
	\caption{Exact powers of the unregularized test (\ref{plot:unreg}), the oracle test (\ref{plot:oracle}), and empirical power of the adaptive test for $\calZ = L^2$ (\ref{plot:adaptive_L2}), $\calZ = H^{0.51}$ (\ref{plot:adaptive_Ht}), and $\calZ = H^1$ (\ref{plot:adaptive_H1}) from $100$ samples against the noise level $\sigma$ in the incompatible smooth scenario (S3) for \rev{$a \in \{2,4\}$, }$l = 5/128$, and $\lambda \in \left\{\frac13,\frac23,1\right\}$.}
	\label{fig:incomp_smooth}
\end{figure}
We find that the unregularized test in fact suffers severely from the issue (I1) raised in the introduction, whereas all regularized tests show a way superior power. As an example, it can be read off of the plots that the power of $\Psi_0$ increases to roughly $50\%$ at a noise level which is smaller by $2$--$3$ orders of magnitude compared to the regularized tests. 
\rev{For a higher level $a = 4$ of ill-posedness, this effectt is even more severe, revealing that the unregularized test is no longer useful, whereas the oracle and the adaptive test show a slightly worse, but still very good performance compared to the case $a = 2$.}
Once again, the optimal test $\Psi_{\Phi^\dagger, c^\dagger}$ displays the best power, and all adaptive tests pay a certain price for not knowing $\udag$. It is slightly surprising that the test for $t = 0.51$ is the best out of the three considered adaptive tests, and it has even a better power than the plain $L^2$-test ($t = 0$). One interpretation of this result is that choosing $\calZ$ as a Sobolev space in fact stabilizes the minimization of the functional $\hat J_Y^{\calZ}$, but if $t$ is chosen too large, this necessarily leads to a smoother, and hence potentially sub-optimal, choice of $\Phi$.

Overall it can be said that regularized testing in fact resolves both issues (I1) and (I2) raised in the introduction. Especially in the incompatible smooth scenario (S3), which is maybe closest to practical applications where smoothness of the feature functional $\varphi \in \X$ seems artificial, the improvement in power is outstanding.

To complete the picture, we have also investigated the performance of our minimization algorithm for $\hat J_Y^{\calZ}$. According to our theory, it might be possible that --- depending on the realization $Y$ of the data --- no minimizer exists.
\rev{However, we have found that a minimizer of $\hat J_Y^{\calZ}$ does (nearly) always exist for the noise levels considered in this study.}

\section{Conclusion and outlook}
\label{sec:conclusion}

We have seen that for linear functionals $\phi \in \overline{\ran T^*}$, an optimal probe functional to test the feature $\langle\phi,\udag\rangle$ using a priori information about the truth $\udag$ exists, and can be characterized as minimizer of the objective functional $J_{T\udag}^\Y$. This optimal regularized test is always superior in terms of its power compared to the unregularized test with the same level.

An adaptive test for such a feature can be constructed by solving a constrained convex optimization problem. We have shown that a solution for this problem exists with a positive probability and have given a lower bound for the power of the adaptive test.

In numerical simulations of a deconvolution problem, we have observed that the construction of adaptive regularized tests using the optimization approach overcomes both issues stated in the beginning: It extends the class of features that can be tested and has a regularizing effect in the choice of the probe functional.
If the problem is sufficiently ill-posed, the adaptive test allows feature testing with a reasonable power in noise regimes where the unregularized test has no power.

Future research may look into ways to avoid the necessity of the optimization approach for two independent data samples, which may not be available in practice.

\rev{This work poses several open questions for future research. Despite the fact that any discussed plug-in test as in \cref{reg_plug_in} will never have better power than the optimal test $\Psi_{\Phi^\dagger, c^*}$, those tests might be of interest in applications because only the regularization parameter $\beta > 0$ (and not an infinite-dimensional probe element $\Phi \in \Y$) needs to be chosen. However, if this is done based on the data $Y$, it is again not clear whether the level $\alpha$ is sustained. It is an interesting question how to design adaptive plug-in tests and to investigate how they perform in practice.}

\rev{Future research should also look into ways to avoid the necessity of two independent data samples, which may not be available in practice. Potential approaches might lie in adaptive plug-in tests as discussed above or Bayesian approaches to testing.}

\section*{Acknowledgments} 

We want to thank the anonymous reviewer and editors for their valuable comments and suggestions.
Remo Kretschmann and Frank Werner are supported by the German Research Foundation (DFG) under the grant WE 6204/2-1.


\appendix

\section{Implementation of convolution}
\label{sec:impl_conv}

\rev{Here, we consider the case of general dimension $d \in \N$. Let $T$ be the convolution operator associated with a kernel $h \in L^1(\bbR^d)$,}
\begin{equation}
	\label{eq:def_conv_op}
	(Tu)(x) := (h * u)(x) = \int_{\bbR^d} h(x - z) u(z) \di z \quad \text{for all}~x \in \bbR^d~\text{and}~u \in L^2(\bbR^d).
\end{equation}
We assume that the kernel $h$ is given in terms of its Fourier transform
\begin{equation*}
	(\Fourier h)(\xi) = \int_{\bbR^d} h(x) \exp \left(2\pi i\scalprod{x,\xi}{}\right) \di x \quad \text{for all}~\xi \in \bbR^d.
\end{equation*}
We approximate $h * u$ by the \emph{periodic convolution}
\[ \left(\widetilde{T}\tilde{u}\right)(x) := (\tilde{h} *_P \tilde{u})(x) = \int_{B^\infty(P/2)} \tilde{h}(x - z)\tilde{u}(z) \di z \]
between the $P$-\emph{periodization} $\tilde{h}$ of $h$,
\begin{equation}
	\label{def_periodization}
	\tilde{h}(x) := h_{\text{per},P}(x) := \sum_{l \in \bbZ^d} h(x + lP) \quad \text{for all}~x \in \bbR^d,
\end{equation}
and the $P$-periodization $\tilde{u} := u_{\text{per},P}$ of $u$, where we assume that the series in \eqref{def_periodization} converges uniformly absolutely.
\rev{We assume that}
\[ \supp u \subseteq \overline{B^\infty(P_u/2)} \]
for some $P_u > 0$, where
\[ B^\infty(r) := \left\{ x \in \bbR^d: \norm{x}{\infty} < r \right\} \]
and that only the values of $h * u$ in $B^\infty(P_y/2)$, $P_y > 0$, are accessible. In this case, we periodize with period $P \ge P_u + P_y$. \rev{This is motivated by the following observation.}

\begin{lem}
	\label{per_conv_exact}
	If $\supp u \subseteq B^\infty(P_u/2)$ and $P \ge P_u + P_y$, then
	\[ \tilde{h} * u = \tilde{h} *_P \tilde{u} \quad \text{on}~B^\infty(P_y/2), \]
	where $\tilde{h} = h_{\text{per},P}$ and $\tilde{u} = u_{\text{per},P}$.
\end{lem}
\begin{proof}
	We obtain
	\begin{align*}
		(\tilde{h} * u)(x) &= \int_{\bbR^2} \tilde{h}(x - z)u(z) \di z
		= \int_{B^\infty(P_u/2)} \tilde{h}(x - z)u(z)\di z \\
		&= \int_{B^\infty(P_u/2)} \tilde{h}(x - z)\tilde{u}(z)\di z
		= \int_{B^\infty(P/2)} \tilde{h}(x - z)\tilde{u}(z)\di z
		= (\tilde{h} *_P \tilde{u})(x)
	\end{align*}
	for all $x \in B^\infty(P_y/2)$. Here, we used that $\supp \tilde{u} \cap B^\infty(P/2) = \supp u \subseteq B^\infty(P_u/2)$.
\end{proof}
Note that $\tilde{h} * u$ and $\tilde{h} *_P \tilde{u}$ do, in general, not agree outside of $B^\infty(P_y/2)$.
\rev{We discretize the problem using the grid $\frac{P}{N}\bbZ_N^d$, where
\[ \bbZ_N := \left\{-\frac{N}{2}, -\frac{N}{2} + 1, \dots, \frac{N}{2} - 1\right\}. \]}
Let $\underline{\tilde{u}}_N$ denote the $N$-periodic sequence
\[ \left(\underline{\tilde{u}}_N\right)_k := \tilde{u}\left(\frac{P}{N}k\right) \quad \text{for all}~k \in \bbZ^d, \]
and $\underline{h} *_N \underline{u}$ the \emph{discrete convolution}
\[ \left(\underline{h} *_N \underline{u}\right)(l) = \sum_{j \in \bbZ_N^d} \underline{h}(l - j) \underline{u}(j), \quad l \in \bbZ^d, \]
between two $N$-periodic sequences $\underline{h}$ and $\underline{u}$.
\rev{We approximate $\underline{\widetilde{T}\tilde{u}}_N = \underline{\tilde{h} *_P \tilde{u}}_N$ by the discrete convolution operator}
\[ T_N \underline{\tilde{u}}_N := P^d\left(\underline{\bar{h}}_N *_N \underline{\tilde{u}}_N\right), \]
where $\bar{h} := P_{N,P}\tilde{h}$ denotes the $L^2$-orthogonal projection of $\tilde{h}$ onto
\[ \mathcal{T}_{N,P} := \spn \left\{\exp \left(\frac{2\pi i}{P}\scalprod{j,\cdot}{}\right): j \in \bbZ_N^d\right\}. \]
We express the discrete convolution as
\begin{equation}
	\label{disc_Four_conv}
	\underline{\bar{h}}_N *_N \underline{\tilde{u}}_N = \DFT_N^{-1} \left(\DFT_N \left(\underline{\bar{h}}_N *_N \underline{\tilde{u}}_N\right)\right) = \text{DFT}_N^{-1} \left(\DFT_N \underline{\bar{h}}_N \DFT_N \underline{\tilde{u}}_N\right)
\end{equation}
using the Fourier convolution theorem, where
\[ \left(\DFT_N \underline{\bar{h}}_N\right)_k = \sum_{j \in \bbZ_N^d} \exp\left(-\frac{2\pi i}{N} \scalprod{k,j}{}\right) (\underline{\bar{h}}_N)_j \quad \text{for all}~k \in \bbZ^d. \]
\rev{We compute the discrete Fourier transform of $\underline{\bar{h}}_N$ analytically.}
\begin{lemma}
	\label{expr_Four_kern}
	\rev{If the series in \eqref{def_periodization} converges uniformly absolutely, then}
	\begin{equation}
		\label{disc_Four_kern}
		(\DFT_N \underline{\bar{h}}_N)_k = \left(\frac{N}{P}\right)^d (\Fourier h)\left(\frac{k}{P}\right) \quad \text{for all}~k \in \bbZ_N^d.
	\end{equation}
\end{lemma}
\begin{proof}
By the aliasing formula [Theorem 4.67, Plonka et al 2018], we now have
\[ (\DFT_N \underline{\bar{h}}_N)_k = N^d \sum_{l \in \bbZ^d} \widehat{\bar{h}}(k + Nl) \quad \text{for all}~k \in \bbZ^d. \]
By Poisson's summation formula, which holds due to the uniform absolute convergence of the series in \eqref{def_periodization}, we moreover have
\begin{equation*}
	\widehat{\bar{h}}(k) = \widehat{P_{P,N}h_{\text{per},P}}(k) = \widehat{h_{\text{per},P}}(k) = P^{-d}(\Fourier h)\left(\frac{k}{P}\right) \quad \text{for all}~k \in \bbZ_N^d,
\end{equation*}
and $\widehat{\bar{h}}(k) = \widehat{P_{P,N}h_{\text{per},P}}(k) = 0$ for $k \in \bbZ^d \setminus \bbZ_N^d$. \end{proof}

Now, we implement the computation of the discrete convolution operator $T_N$ using \eqref{disc_Four_conv}, \eqref{disc_Four_kern}, and a fast Fourier transform.
\rev{The following lemma shows that $T_N\underline{\tilde{u}}_N$ is in fact a discretization of the periodic convolution $\bar{h} *_P (Q_{N,P}\tilde{u})$,} where $Q_{N,P}\tilde{u}$ denotes the interpolation of $\tilde{u}$ in the grid points $\frac{P}{N}j$, $j \in \bbZ_N^d$, by a function in $\mathcal{T}_{N,P}$.
\begin{lem}
	We have
	\[ \underline{\bar{h} *_P \left(Q_{N,P}\tilde{u}\right)}_N = P^d \left(\underline{\bar{h}}_N *_N \underline{\tilde{u}}_N\right). \]
\end{lem}
\begin{proof}
	For $f \in \mathcal{T}_{N,P}$ we have
	\[ \widehat{f}(k) = \begin{cases}
		(\DFT_N \underline{f}_N)(k) & \text{for}~k \in \bbZ_N^d, \\
		0 & \text{otherwise}.
	\end{cases} \]
	As $Q_{N,P}\tilde{u} \in \mathcal{T}_{N,P}$, it follows from the periodic convolution theorem that
	\[ \widehat{\left(\bar{h} *_P (Q_{N,P})\tilde{u}\right)}(k) = P^d \widehat{\bar{h}}(k) \widehat{Q_{N,P}\tilde{u}}(k) = 0 \quad \text{for all}~k \in \bbZ^d \setminus \bbZ_N^d, \]
	which implies that $\bar{h} *_P (Q_{N,P})\tilde{u} \in \mathcal{T}_{N,P}$ as well.
	Now, the discrete convolution theorem yields
	\begin{equation*}
		\begin{split}
			\DFT_N \left(\underline{\bar{h} *_P \left(Q_{N,P}\tilde{u}\right)}_N\right)(k)
			&= \widehat{\left(\bar{h} *_P (Q_{N,P})\tilde{u}\right)}(k)
			= P^d \widehat{\bar{h}}(k) \widehat{Q_{N,P}\tilde{u}}(k) \\
			&= P^d \left(\DFT_N \underline{\bar{h}}_N\right)(k) \left(\DFT_N \underline{\bar{u}}_N\right)(k) \\
			&= P^d \DFT_N \left(\underline{\bar{h}}_N *_N \underline{\tilde{u}}_N\right)(k)
			\quad \text{for all}~k \in \bbZ_N^d.
		\end{split}
	\end{equation*}
	The statement follows from the injectivity of the discrete Fourier transform.
\end{proof}

\rev{Moreover, the following identity holds.}
\begin{lem}
	\label{proj_kern}
	We have
	\[ \tilde{h} *_P (Q_{N,P}\tilde{u}) = \bar{h} *_P (Q_{N,P}\tilde{u}), \]
	where $\bar{h} := P_{N,P}\tilde{h}$.
\end{lem}
\begin{proof}
	As $\widehat{Q_{N,P}\tilde{u}}(k) = 0$ for $k \in \bbZ^d \setminus \bbZ_N^d$, we have
	\begin{equation*}
		\begin{split}
			\widehat{\tilde{h} *_P (Q_{N,P}\tilde{u})} &= P^d \widehat{\tilde{h}} \widehat{Q_{N,P}\tilde{u}} = P^d \widehat{P_{N,P}\tilde{h}} \widehat{Q_{N,P}\tilde{u}} = \widehat{(P_{N,P}\tilde{h}) *_P (Q_{N,P}\tilde{u})}
		\end{split}
	\end{equation*}
	by the periodic convolution theorem. Now, the statement follows from the injectivity of the periodic Fourier transform.
\end{proof}

\rev{Now, we bound the remaining interpolation error between $\tilde{h} *_P \tilde{u}$ and $\tilde{h} *_P (Q_{N,P}\tilde{u})$ as well as the periodization error in the following overall error estimate.}
\begin{thm}
	\label{conv_err_est}
	If $u \in H^m(\bbR^d)$, $m \in \N$, $\supp u \subseteq B^\infty(P_u/2)$, and $P \ge P_u + P_y$, then the approximation error is bounded by
	\begin{equation*}
		\begin{split}
			\norm{\bar{h} *_P (Q_{N,P}\tilde{u}) - h * u}{L^2(B^\infty(P_y/2))}
			&\le \norm{h - \tilde{h}}{L^1(B^\infty(P/2))} \norm{u}{L^2(\bbR^d)} \\
			&\quad + \norm{h}{L^1(\bbR^d)} \norm{Q_{N,P} - I}{H^m(B^\infty(P/2)) \to L^2(B^\infty(P/2))} \norm{u}{H^m(\bbR^d)}.
		\end{split}
	\end{equation*}
\end{thm}
\begin{proof}
	For $g \in L^1(B^\infty(P/2))$ and $f \in L^2(B^\infty(P/2))$, the estimate
	\[ \norm{g *_P f}{L^2} \le \norm{g}{L^1} \norm{f}{L^2} \]
	holds.
	It follows from \cref{per_conv_exact} that the periodization error is bounded by
	\begin{equation*}
		\begin{split}
			\norm{h * u - \tilde{h} *_P \tilde{u}}{L^2(B^\infty(P_y/2))} &= \norm{h * u - \tilde{h} * u}{L^2(B^\infty(P_y/2))} \\
			&\le \norm{h - \tilde{h}}{L^1(B^\infty(P_y/2))} \norm{u}{L^2(B^\infty(P_y/2))},
		\end{split}
	\end{equation*}
	and from \cref{proj_kern} that the interpolation error is bounded by
	\begin{equation*}
		\begin{split}
			\norm{\bar{h} *_P (Q_{N,P} \tilde{u}) - \tilde{h} *_P \tilde{u}}{L^2(B^\infty(P/2))}
			&= \norm{\tilde{h} *_P (Q_{N,P} \tilde{u}) - \tilde{h} *_P \tilde{u}}{L^2(B^\infty(P/2))} \\
			&\le \smallnorm{\tilde{h}}{L^1(B^\infty(P/2))} \norm{(Q_{N,P} - I)\tilde{u}}{L^2(B^\infty(P/2))}.
		\end{split}
	\end{equation*}
	The triangle inequality yields
	\begin{align*}
		&\norm{\bar{h} *_P (Q_{N,P}\tilde{u}) - h * u}{L^2(B^\infty(P_y/2))} \\
		&\le \norm{\bar{h} *_P (Q_{N,P}\bar{u}) - \tilde{h} *_P \tilde{u}}{L^2(B^\infty(P/2))} + \norm{\tilde{h} *_P \tilde{u} - h * u}{L^2(B^\infty(P_y/2))} \\
		&\le \norm{h - \tilde{h}}{L^1(B^\infty(P/2))} \norm{u}{L^2} \\
		&\quad + \norm{h}{L^1(B^\infty(P/2))} \norm{Q_{N,P} - I}{H^m(B^\infty(P/2)) \to L^2(B^\infty(P/2))} \norm{u}{H^m(\bbR^d)}
	\end{align*}
	where we used that $\smallnorm{\tilde{h}}{L^1(B^\infty(P/2))} = \norm{h}{L^1(\bbR^d)}$.
\end{proof}

\rev{The periodization error of the kernel can, moreover, be controlled by choosing $P$ large enough.}
\begin{lem}
	For any $h \in L^1(\bbR^d)$ and $\epsilon > 0$ there exists $P > 0$ such that
	\[ \norm{h - h_{\text{per},P}}{L^1(B^\infty(P/2))} \le \epsilon. \]
\end{lem}
\begin{proof}
	Since $h \in L^1(\bbR^d)$, we can choose $P$ large enough such that
	\begin{equation*}
		\norm{h(x) - h_{\text{per}, P}(x)}{L^1(B^\infty(P/2))}
		= \norm{\sum_{n \in \bbZ^d \setminus \{0\}} h(\cdot + nP)}{L^1(B^\infty(P/2))}
		\le \int_{\bbR^d \setminus B^\infty(P/2)} \abs{h(x)} \di x
		< \epsilon. \qedhere
	\end{equation*}
\end{proof}

Last of all, we estimate the interpolation error for the $2$-periodized problem in $d = 1$ with the specific kernel considered in \cref{sec:num}.
\begin{theorem}
	Let $h \in L^1(-1,1)$ be defined by \eqref{eq:def_h}. If $u \in H^1(\bbR)$, $\supp u \subseteq [-\frac12, \frac12]$, and $P = 2$, then the interpolation error is bounded by
	\begin{equation*}
		\norm{\widetilde{T}\tilde{u} - \bar{h} *_P (Q_{N,2}\tilde{u})}{L^2(-1,1)}
		\le 48 N^{-1} \norm{\tilde{u}}{H^1(-1,1)}.
	\end{equation*}
\end{theorem}
\begin{proof}
By definion of $h$, we have
\[ \smallnorm{\tilde{h}}{L^1(-1,1)} = \norm{h}{L^1(\bbR)} = \int_{\bbR} h(x) \di x = (\Fourier h)(0) = 1. \]
By Corollary 2.47 in \cite{proessdorf1991}, the interpolation operator $Q_{N,P}$ satisfies
\[ \norm{I - Q_{N,2}}{H^t(-1,1) \to H^s(-1,1)} \le 48 N^{s - t} \]
for $s \in (0,1]$ and $t > 1$.
Taking the limit $s \to 0$, setting $t = 1$, using \cref{proj_kern}, and proceeding as in the proof of \cref{conv_err_est} leads to the estimate
\begin{align*}
	\norm{\widetilde{T}\tilde{u} - \bar{h} *_P (Q_{N,2}\tilde{u})}{L^2(-1,1)}
	&= \norm{\tilde{h} *_P \tilde{u} - (P_{N,2}\tilde{h}) *_P (Q_{N,2}\tilde{u})}{L^2(-1,1)} \\
	&\le \smallnorm{\tilde{h}}{L^1(-1,1)} \norm{(Q_{N,2} - I)\tilde{u}}{L^2(-1,1)}
	\le 48 N^{-1} \norm{\tilde{u}}{H^1(-1,1)} 
\end{align*}
for all $\tilde{u} \in H^1(-1,1)$.
\end{proof}

\section{Implementation of minimization}
\label{sec:impl_min}

For a given $y \in \calZ^* = H^{-t}(-1,1)$, we express the objective functional as
\begin{equation*}
	\begin{split}
		\hat{J}_y^{H^t}(e,s) + \delta_U(e,s) &= \norm{T^*e - s\phi}{L^\infty} - \scalprod{y,e}{H^{-t} \times H^t} + \delta_{B_1^{H^t}}(e) + \delta_{\bbR_+}(s) \\
		&= G(K(e,s)) + F(e,s),
	\end{split}
\end{equation*}
where $F$: $H^t(-1,1) \times \bbR \to \overline{\bbR}$,
\[ F(e,s) := \scalprod{y, e}{H^{-t} \times H^t} + \delta_{B_1^{H^t}}(e) + \delta_{\bbR_+}(s), \]
$G$: $L^2(-1,1) \to \overline{\bbR}$,
\[ G(r) := \norm{r}{L^\infty}, \]
and $K$: $H^t(-1,1) \times \bbR \to L^2(\bbR)$,
\[ K(e,s) = T^* e - s\phi. \]
We consider the constrained convex optimization problem
\begin{equation}
	\label{eq:PDPS_min_prob}
	\min_{x \in H^t(-1,1) \times \bbR} F(x) + G(Kx),
\end{equation}
where we denote $x = (e,s)$.
We solve this problem numerically using the \emph{primal-dual proximal splitting} (PDPS or \emph{Chambolle--Pock}) \emph{method}
\begin{equation*}
	\left\{
		\begin{aligned}
			x^{k + 1} &= \prox_{\tau F} (x^k - \tau K^* r^k), \\
			\bar{x}^{k+1} &= 2x^{k+1} - x^k, \\
			r^{k + 1} &= \prox_{\rho G^*} (r^k + \rho K \bar{x}).
		\end{aligned}
	\right.
\end{equation*}
see \cite{ChambollePock2011} and \cite[Section 8.4]{clason2020}. 
We note that $F$ and $G$ are proper, convex, and lower semicontinuous and $K$ is bounded and linear.
As the initial guess of the primal variable we use $x^0 = (e^0,s^0)^\mathrm{T}$ with
\[ s^0 = \frac{1}{\norm{\Phi_0}{H^t}} \quad\text{and}\quad e^0 = \frac{\Phi_0}{\norm{\Phi_0}{H^t}}, \]
where $\Phi_0$ denotes the probe functional of the unregularized test, i.e., it satisfies $T^*\Phi_0 = \phi$.
We then initialize the dual variable by $r^0 = T^* e^0 - s^0 \phi$.
By Corollary 9.13 in \cite{clason2020}, the sequence $(x^k,r^k)_{k\in\N}$ generated by the PDPS method converges weakly to a pair satisfying the Fenchel extremality conditions of the minimization problem \eqref{eq:PDPS_min_prob} if the condition
\begin{equation}
	\label{eq:PDPS_conv_cond}
	\tau\rho\norm{K}{H^t \times \bbR \to L^2}^2 < 1
\end{equation}
is satisfied.
We endow $H^t(-1,1) \times \bbR$ with the norm
\[ \norm{(e,s)}{H^t \times \bbR} = \left(\norm{e}{H^t}^2 + \abs{s}^2\right)^\frac12. \]
\begin{lem}
	The operator norm of $K$ satisfies
	\[ \norm{K}{H^t \times \bbR \to L^2} \le 2. \]
\end{lem}
\begin{proof}
For all $(e,s) \in H^t(-1,1) \times \bbR$, we have
\begin{equation*}
	\begin{split}
		\norm{K(e,s)}{L^2} &= \norm{\widetilde{T}^*e - s\phi}{L^2}
		\le \norm{\widetilde{T}^*}{H^t \to L^2} \norm{e}{H^t} + \abs{s} \norm{\phi}{L^2} \\
 		&\le \left(\norm{\widetilde{T}^*}{H^t \to L^2} + \norm{\phi}{L^2}\right) \norm{(e,s)}{H^t \times \bbR}.
	\end{split}
\end{equation*}
Now,
\begin{equation*}
	\begin{split}
		\norm{\widetilde{T}^*e}{L^2}^2 &= \norm{\tilde{h} * e}{L^2}^2 = P\norm{\Fper\left(\tilde{h} * e\right)}{\ell^2}^2 = P^2\norm{\Fper\tilde{h} \cdot \Fper e}{\ell^2}^2 \\
		&\le \norm{\rev{\left(1 + \frac{0.03^2}{P^2}k^2\right)^{-2a}}\left(1 + k^2\right)^{-t}}{\ell^\infty} \norm{\left(1 + k^2\right)^\frac{t}{2} (\Fper e)(k)}{\ell^2}^2
		= \norm{e}{H^t}^2
	\end{split}
\end{equation*}
for all $e \in H^t(-1,1)$, so that $\norm{\widetilde{T}^*}{H^t \to L^2} \le 1$, and consequently
\begin{equation*}
	\norm{K}{H^t \times \bbR \to L^2} \le \norm{\widetilde{T}^*}{H^t \to L^2} + \norm{\phi}{L^2} \le 1 + 1 = 2. \qedhere
\end{equation*}
\end{proof}

With this knowledge, we choose the parameters $\tau = \rho = 0.25 < 1/2 = \norm{K}{H^t \times \bbR \to L^2}^{-1}$. This way, condition \eqref{eq:PDPS_conv_cond} is satisfied.

\begin{lem}
	The adjoint $K^*$: $L^2(-1,1) \to H^t(-1,1) \times \bbR$ of $K$ is given by
	\[ K^*(r) = \begin{pmatrix}
		R\rev{\widetilde{T}r} \\
		-\scalprod{\phi,r}{L^2}
	\end{pmatrix}, \]
	where $R$ denotes the Riesz isomorphism between $H^{-t}(-1,1)$ and $H^t(-1,1)$.
\end{lem}
\begin{proof}
For $(e,s) \in H^t(-1,1) \times \bbR$ and $r \in L^2(-1,1)$ we have
\begin{equation*}
	\begin{split}
		\scalprod{K(e,s), r}{L^2} &= \scalprod{\widetilde{T}^*e - s\phi, r}{L^2}
		= \scalprod{e, \widetilde{T}r}{L^2} - s\scalprod{\phi,r}{L^2}
		= \scalprod{\rev{\widetilde{T}r},e}{H^{-t} \times H^t} - s\scalprod{\phi,r}{L^2} \\
		&= \scalprod{R\rev{\widetilde{T}r},e}{H^t} - s\scalprod{\phi,r}{L^2}
		= \scalprod{\begin{pmatrix}R\rev{\widetilde{T}r} \\ -\scalprod{\phi,r}{L^2}\end{pmatrix}, \begin{pmatrix}e \\ s\end{pmatrix}}{H^t \times \bbR}. \qedhere
	\end{split}
\end{equation*}
\end{proof}

For $t \in \bbR$, we define the sequence space
\[ \ell_t^2(\bbZ) := \left\{f: \bbZ \to \bbR \Bigg\vert k \mapsto (1 + k^2)^\frac{t}{2} f(k) \in \ell^2(\bbZ)\right\}. \]
\begin{lem}
	The Riesz isomorphism $R$ between $H^{-t}(-1,1)$ and $H^t(-1,1)$ can be expressed as
	\[ R = \left(\Fper\right)^{-1} M_{-t} \Fper, \]
	where $M_t$: $\ell_t^2(\bbZ) \to \ell_{-t}^2(\bbZ)$ denotes the multiplication operator
	\[ (M_t f)(k) := \left(1 + k^2\right)^t f(k), \]
	and $\Fper$: $H^{-t}(-1,1) \to \ell_{-t}^2(\bbZ)$ the unique extension of the periodic Fourier transform on $L^2(-1,1)$.
\end{lem}
\begin{proof}
For all $y \in L^2(-1,1)$ and $e \in H^t(-1,1)$, we have
\begin{equation*}
	\begin{split}
		P\scalprod{\Fper y, \Fper e}{\ell^2} &= \scalprod{y,e}{L^2} = \scalprod{y,e}{H^{-t} \times H^t} \\
		&= \scalprod{Ry,e}{H^t}	= P \sum_{k \in \bbZ} \left(1 + k^2\right)^t (\Fper Ry)(k) (\Fper e)(k)
	\end{split}
\end{equation*}
by the Plancherel theorem and the Riesz representation theorem.
Due to the isometry of $P^{1/2}\Fper$, the density of $H^t(-1,1)$ in $L^2(-1,1)$ also implies the density of $\Fper(H^t(-1,1))$ in $\ell^2(\bbZ)$. Thus,
\[ \Fper y = M_t \Fper R y \]
for all $y \in L^2(-1,1)$.
This yields
\[ R\rvert_{L^2(-1,1)} = (\Fper)^{-1}M_t^{-1}\Fper = (\Fper)^{-1}M_{-t}\Fper. \]
Since $L^2(-1,1)$ is dense in $H^{-t}(-1,1)$ and $R$ is isometric (and thus bounded), it follows that $R$ is the unique extension of the operator $(\Fper)^{-1}M_{-t}\Fper$ on $L^2(-1,1)$ to $H^{-t}(-1,1)$. Now, the statement follows from the uniqueness of the extension of the Fourier transform.
\end{proof}

\begin{lem}
	The proximal point mappings of $\tau F$ and $\rho G^*$ are given by
	\begin{equation*}
		\begin{aligned}
			\prox_{\tau F}(e,s) &= \proj_U^{H^t \times \bbR} (e + \tau Ry, s), \\
			\prox_{\rho G^*}(r) &= \proj_{B_1^{L^1}}(r).
		\end{aligned}
	\end{equation*}
\end{lem}
\begin{proof}
We compute
\begin{equation*}
	\begin{split}
		\prox_{\tau F}(e,s) &= \argmin_{(z_e,z_s) \in U} \frac12\norm{(z_e,z_s) - (e,s)}{H^t \times \bbR}^2 + \tau\scalprod{y,e}{H^{-t} \times H^t} \\
		&= \argmin_{(z_e,z_s) \in U} \frac12\norm{z_e - e}{H^t}^2  + \frac12\abs{z_s - s}^2 - \tau\scalprod{Ry, z_e - e}{H^t} \\
		&\quad - \tau\scalprod{Ry, e}{H^t} + \frac{\tau^2}{2}\norm{Ry}{H^t}^2 - \frac{\tau^2}{2}\norm{Ry}{H^t}^2 \\
		&= \argmin_{(z_e,z_s) \in U} \frac12\norm{z_e - e - \tau Ry}{H^t}^2 + \frac12\abs{z_s - s}^2
		= \proj_U^{H^t \times \bbR} (e + \tau Ry, s).
	\end{split}
\end{equation*}

The Fenchel conjugate of $G = \norm{\cdot}{L^\infty(-1,1)}$ is given by
\[ G^* = \delta_{B_1^{L^1(-1,1)}}, \]
the indicator function of the unit ball in $L^1(-1,1)$, see Example 5.3 in \cite{clason2020}.
Therefore,
\[ \prox_{\rho G^*} = \prox_{\delta_{B_1^{L^1(-1,1)}}} = \proj_{B_1^{L^1(-1,1)}}. \qedhere \]
\end{proof}

Given a vector with the values of a function in the grid points $\bbZ_N$, we discretize this projection by the projection onto the $1$-norm unit ball. We implement this projection using an algorithm introduced in \cite{held1974}. An overview over different algorithms for the projection to the $1$-norm ball can be found in \cite{condat2016}.

\section{Normalized discretization}
\label{sec:norm_disc}

\rev{In our numerical simulations, we discretize the norms and inner products on $L^1(-1,1)$, $L^2(-1,1)$, and $L^\infty(-1,1)$ according to
\begin{align*}
	\norm{v}{L^1(-1,1)} &\approx \frac{2}{n} \norm{\underline{v}_N}{1}, &
	\norm{v'}{L^\infty(-1,1)} &\approx \norm{\underline{v}_N^\prime}{\infty} \\
	\norm{x}{L^2(-1,1)} &\approx \sqrt{\frac{2}{n}} \norm{\underline{x}_N}{2} &
	\scalprod{x, x'}{L^2(-1,1)} &\approx \frac{2}{n} \scalprod{\underline{x}_N, \underline{x}_N^\prime}{2},
\end{align*}
where $(\underline{x}_N)(k) = x(2k/N)$ for all $k \in \bbZ_N$.
We furthermore work with scaled versions 
\[ \tilde{u}^\dagger := \frac{2}{n}\udag, \quad \tilde{\phi} := \sqrt{\frac{2}{n}}\phi, \quad \tilde{\Phi} := \sqrt{\frac{2}{n}}\Phi \]
of $\udag$, $\phi$, and $\Phi$. These satisfy
\[ \norm{\underline{\tilde{u}}_N^\dagger}{1} = 1 \quad\text{and}\quad \norm{\underline{\tilde{\phi}}_N}{2} = 1. \]
Moreover, we have
\begin{align*}
	J_{T\udag}^\Y(\Phi) &= \frac{\norm{T^*\Phi - \phi}{L^\infty} - \scalprod{T\udag, \Phi}{L^2}}{\norm{\Phi}{L^2}}
	= \frac{\sqrt{\frac{n}{2}}\norm{T^*\tilde{\Phi} - \tilde{\phi}}{L^\infty} - \frac{n}{2}\sqrt{\frac{n}{2}}\scalprod{T\tilde{u}^\dagger, \tilde{\Phi}}{L^2}}{\sqrt{\frac{n}{2}}\norm{\tilde{\Phi}}{L^2}} \\
	&\approx \sqrt{\frac{n}{2}} \cdot \frac{\norm{\underline{T^*\tilde{\Phi}}_N - \underline{\tilde{\phi}}_N}{\infty} - \scalprod{\underline{T\tilde{u}^\dagger}_N, \underline{\tilde{\Phi}}_N}{2}}{\norm{\underline{\tilde{\Phi}}_N}{2}}.
\end{align*}
We use this expression to compute the power of the considered regularized tests.}


\bibliography{literature}

\end{document}